\documentclass[a4paper,11pt,reqno]{amsart}
\usepackage{amsmath,amsthm,amssymb,graphicx,subfigure}
\usepackage{bbm} \usepackage{graphicx,verbatim}
\usepackage{color}
\usepackage [english]{babel}
\usepackage [latin1] {inputenc}
\usepackage [T1] {fontenc}
\usepackage{amssymb}
\usepackage{amsfonts} 
\usepackage{dsfont}
\usepackage{amsmath}
\usepackage{stmaryrd}
\usepackage{graphicx} 
\usepackage{fullpage}
\usepackage{mathpazo}
\newtheorem{theorem}{Theorem}[section]

\newtheorem{proposition}[theorem]{Proposition}
\newtheorem{corollary}[theorem]{Corollary}

\newtheorem{conjecture}[theorem]{Conjecture}
\newtheorem{remark}[theorem]{Remark}

\date{}
\begin{document}

\title[Convergence of random holomorphic functions ]{Convergence of random holomorphic functions with real zeros and extensions of the stochastic zeta function}
\author[Joseph Najnudel]{{Joseph} Najnudel}
\address{School of Mathematics, University of Bristol, UK}
\email{joseph.najnudel@bristol.ac.uk}
\author[Ashkan Nikeghbali]{{Ashkan} Nikeghbali}
\address{Institut f\"ur Mathematik, Universit\"at Z\"urich} \email{ashkan.nikeghbali@math.uzh.ch}

\date{\today} 
\begin{abstract}
In this article, we provide a unified framework for studying the convergence of rescaled characteristic polynomials of random matrices from various classical ensembles as well as functional convergence results for the Riemann zeta function.  To this end, we consider the more general viewpoint of converging point processes (a special case of which is the sequence of converging eigenvalue point processes from random matrix ensembles), and we identify sufficient conditions under which the convergence of random point processes on the real line implies the convergence in law, for the topology of uniform convergence on compact sets, of suitable random holomorphic functions whose zeros are given by the point processes which are considered.   Our results extend convergence results for rescaled characteristic polynomials obtained by various authors (in the case of the circular unitary ensemble, the limiting random analytic function is called the stochasic zeta function). We also show that for a wide class of point processes associated with these limiting random holomorphic functions (we can often interpret these points as the spectrum of some random operator), their Stieltjes transform follows for almost all points of the real line the standard Cauchy distribution, reminiscent of the results by Aizenman and Warzel (\cite{AW15}) in the case of the sine kernel point process.
\end{abstract}

\maketitle



\tableofcontents
\newpage
\section{Introduction}

The past two decades have witnessed  tremendous developments in the so called random matrix approach in number theory, which originated in the celebrated Montgomery conjecture, further extended by Rudnik and Sarnak \cite{RS96} and known today as the GUE conjecture,  which asserts that after proper rescaling, the ordinates of the zeroes of the Riemann zeta function on the critical line behave as a sine kernel point process, which can be obtained  for instance as the scaling limit of eigenangles of random unitary matrices (or matrices from  the Gaussian Unitary Ensemble or in short GUE). Inspired by this "spectral" interpretation for the statistics of the zeros of the Riemann zeta function, Keating and Snaith \cite{KS00} proposed a new probabilistic model to predict several very fine results on the distribution of the values of the Riemann zeta function on the critical line, namely the characteristic polynomial of random unitary matrices:
\begin{equation}
	Z_n(X)=\det (Id-U_n^* X),
\end{equation}where $U_n$ is a matrix from the Circular Unitary Ensemble (in short CUE) of size $n$, i.e. a random matrix from the unitary group $U(n)$ equipped with the Haar measure (normalized to be a probability measure). When $X=e^{i t}$, with $t\in \mathbb R$, $Z_n(X)$ is thought of as a good model for $\zeta(1/2+ i t)$, $t\in\mathbb R$. In particular, computing $\mathbb E[|Z_n(1)|^{2k}]$, for $k\in\mathbb C$, with $\Re k>-1/2$, they were able to conjecture the moments for the Riemann zeta function: $
\frac{1}{T}\int_0^T |\zeta(1/2+it)|^{2k} dt$ when $T\to\infty$ (see \cite{KS00}). Following this work, many new questions and problems were raised, and the characteristic polynomial of a random unitary matrix turned out to be a very rich and deep mathematical object with further connections to number theory (function fields), representation theory, supersymmetry, the Gaussian multiplicative chaos and log-correlated fields, the theory of orthogonal polynomials on the unit circle, etc. 

A few general conceptual questions were then raised over the years, in particular about the existence of natural infinite-dimensional objects "sitting" above all random unitary matrices, and which could be obtained as scaling limits from random matrix theory. Let us mention a few such questions which are of interest to our work in this article:
\begin{itemize}
	\item Given the success of the characteristic polynomial to make predictions for the distribution of values of the Riemann zeta function on the critical line, is there a natural random holomorphic function that could be obtained as some scaling limit from the characteristic polynomial, with all zeroes on some line? If such an object can be found, what is its relation to the Riemann zeta function and the GUE conjecture?
	\item For admissible $s_1,\cdots, s_r, t_1\cdots, t_r \in\mathbb C$, is there a limit, when $n\to\infty$, to the expectation $$\mathbb E \left[ \frac{Z_n(e^{2 i \pi s_1/n})\cdots Z_n(e^{2 i \pi s_r/n})}{Z_n(e^{2 i \pi t_1/n})\cdots Z_n(e^{2 i \pi t_r/n})}\right]?$$ Is there a limit to the ratios inside the expectation? 
\end{itemize}

There were various approaches to these problems and to identify the possible scaling limits mentioned above. In \cite{CNN17} we proved that the natural scaling limit is the  following random analytic function, defined for $s\in \mathbb C$, as 
\begin{equation}
	\xi_\infty(s)=e^{i \pi s} \prod_{k\in\mathbb Z}\left(1-\frac{s}{y_k} \right),
\end{equation}where $(y_k)_{k\in\mathbb Z}$ is a sine kernel point process, and the product being understood as a principal value, i.e.
$$\prod_{k\in\mathbb Z}\left(1-\frac{s}{y_k} \right)=\left(1-\frac{s}{y_0}\right) \lim_{n\to \infty} \prod_{k=1}^n \left(1-\frac{s}{y_k} \right)\left(1-\frac{s}{y_{-k}} \right).$$
This random holomorphic function  naturally appeared in our solution to some of the problems mentioned above:
\begin{itemize}
	\item as a scaling limit of the characteristic polynomial of random unitary matrices: if we note $$\xi_n(s)=\frac{Z_n(e^{2 i \pi s/n})}{Z_n(1)},$$then $\xi_n(s)$ converges to $\xi_{\infty}$ both almost surely (within the setting of virtual isometries) and in distribution on the space of holomorphic functions equipped with the norm of uniform convergence on compact sets;
	\item as a limiting function in the ratio conjecture: we identified that for $s_1,\cdots, s_r, t_1\cdots, t_r \in\mathbb C$ and $t_j\notin (y_k)_{k\in \mathbb Z}$, the ratio 
	\begin{equation}
		\frac{Z_n(e^{2 i \pi s_1/n})\cdots Z_n(e^{2 i \pi s_r/n})}{Z_n(e^{2 i \pi t_1/n})\cdots Z_n(e^{2 i \pi t_r/n})}
	\end{equation} converges in the same sense as above to the ratio 
	\begin{equation}
		R(s_1,\cdots, s_r, t_1\cdots, t_r)=\frac{\xi_{\infty}(s_1)\cdots \xi_{\infty}(s_r)}{\xi_{\infty}(t_1)\cdots \xi_{\infty}(t_r)}
	\end{equation} and we were able to compute $\mathbb E[R(s_1,\cdots, s_r, t_1\cdots, t_r)]$ away from the real line.
\end{itemize}

These results allowed us to make our own predictions for moments of ratios of the Riemann zeta function and of its logarithmic derivative. We also conjectured: 

\begin{conjecture}[The Stochastic Zeta Conjecture]
 If $\omega$ is a uniform random variable on $[0,1]$ and $T>0$, then the following convergence holds, as $T\to \infty$ 
\begin{equation}
	\left(\frac{\zeta\left(\frac{1}{2} + i \omega T-\frac{i 2 \pi s}{\log T}\right)}{\zeta\left(\frac{1}{2} + i \omega T\right)}; s\in\mathbb C  \right)\to \left(\xi_\infty(s); s\in\mathbb C \right)
\end{equation}
in law, uniformly in $s$ on every compact set. 
\end{conjecture}
We then argued that this functional convergence is essentially equivalent to the GUE conjecture itself (this fact was also stated by Sodin in \cite{S18} with a detailed sketch of proof). In our current paper, we give a complete proof, expanding the number-theoretic arguments and showing how they combine with our probabilistic setting, that this conjecture implies the GUE conjecture and that conversely the GUE conjecture (under the Riemann hypothesis and the simplicity of the zeros of the Riemann zeta function) implies this conjecture:

\begin{theorem} \label{stozetaeq}
The following relations between the stochastic zeta conjecture and the GUE conjecture hold:
\begin{enumerate}
\item The stochastic zeta conjecture implies the GUE conjecture;
\item Assuming the Riemann Hypothesis and the simplicity of the zeros of the zeta function, the GUE conjecture implies the stochastic zeta conjecture.
\end{enumerate}

\end{theorem}

Unlike the moments conjecture where both primes and random matrix statistics seem to combine or conspire, the scale and the normalization we use are only described with random matrix statistics (prime numbers do not appear anymore at the limit).

The approach we used in \cite{CNN17} to obtain the stochastic zeta function as a scaled limit of a ratio relied in particular on some coupling of the unitary groups $U(n)$, and some quantitative form of the strong (almost sure) convergence of the eigenangle point processes of these coupled random unitary matrices. We needed to control both local and global behavior of the eigenvalues. After this work, the question of whether our results could be extended to other random matrix ensembles was naturally raised, and more specifically the existence of scaling limits for the characteristic polynomials of other ensembles. In \cite{CHNNR19}, introducing the concept of amenable point processes, we solved the problem for the GUE ensemble (for which, to the best of our knowledge, no coupling resulting in a.s. convergence of the eigenvalues exists) and other compact groups relevant in the random matrix approach in number theory. Lambert and Paquette (\cite{LP20}) recently solved the problem for the edge of the spectrum, and Valk\'o and Vir\'ag in \cite{VV20} proved the existence of such a scaling limit for the circular beta ensembles C$\beta$E, for $\beta>0$: to each such an ensemble is associated a natural random holomorphic function in the scaling limit, with zeros distributed on the real line as the Sine$_\beta$ point process. The case $\beta=2$ corresponds to our $\xi_\infty$ function. They call this function the stochastic zeta function.  Moreover, they show that these functions enjoy remarkable properties, such as a Taylor series expansion, in which coefficients can be expressed thanks to independent two-sided Brownian motions and Cauchy random variables. Using their work on random operators having a Sine$_\beta$ spectrum, they provide a spectral interpretation of these stochastic zeta functions: they can be viewed as the $\det_2(Id-s\tau^{-1})$ for such operators (we shall come back to this later but they use  $\det_2$ because these inverse operators are not trace class but Hilbert-Schmidt, and $\det_2$ amounts to using a Weierstrass factor to make the product representation of an entire function of order $1$ to converge or to using the principal value point of view as we have done). Let us also note that in a recent article \cite{BU24} by Bufetov, the stochastic zeta function appears in the study of Gaussian multiplicative chaos for the sine process. These new works tend to show that these random holomorphic functions are rich and natural limiting objects appearing as a scaling limit for many different models.

In this work, we propose a new general framework based on the convergence of point processes to study the convergence of random holomorphic functions with real zeros. We shall view a sequence of random holomorphic functions as given by their zeros point processes on the line and give general conditions under which such a sequence converges. The main theorem which might seem abstract at first sight will in fact cover all existing convergence results for the various random matrix ensembles mentioned above and will provide us with new examples (in particular very general families of point processes on the unit circle as shown in Section~\ref{examples}). In particular cases which do not fit well in the random matrix interpretation (such as the Sine$_\beta$ for $\beta=0$, i.e. a Poisson point process) will naturally fit this setting. 

\vspace{0.5cm} 

 From now on, if $M$ is a Radon measure on the real line, we denote: 
$$M(x) := M ([0,x]) \mathds{1} _ {x \geq 0} - M((x,0)) \mathds{1}_{x < 0}.$$

The next Theorem is our main result when it comes to understanding the convergence of a sequence of random holomorphic functions given the convergence of their zeros. We shall work in the setting of a sequence of point processes $(X_n)$ which converges in law to a point process $X$ (recall that the topology on the space of point measures is that of vague convergence). Moreover, the points we consider shall satisfy the following property: the sum of their inverses diverges (but converges in the principal value sense), while the sum of their squared inverses converges. For each $X_n$ and $X$ we wish to construct a random holomorphic function whose zeros are the points of $X_n$ and $X$ respectively. These functions will be a.s.  entire functions of order $1$. We can naturally construct such entire functions either as a converging product in principal value or with the help of the Weierstrass factors (Proposition~\ref{weierstrass} shows how to go from one product to another). Once we are able to construct such sequences of holomorphic functions, we would like to be able to translate the convergence of the point processes (and hence the convergence of the zeros) into the convergence in law, for the topology of uniform convergence on compact sets,  of the corresponding entire functions. To achieve all this, we shall associate with the random measures $(X_n)$ and $X$ some deterministic Radon measures $M_{0,n}$ and $M_0$ which will smoothly approximate the expected number of points in compact intervals and which will allow for a deterministic exponential factor compensation so that the product over zeros converges almost surely. As we shall need to control the tails of our point processes, we will be interested in controlling the convergence on all compact intervals simultaneously. We also need these uniform estimates for compact intervals to be also uniform in $n$ in order to have the convergence of the corresponding holomorphic functions. As our point processes $X_n$ converge to $X$, we will also require that $M_{0,n}$ converge vaguely to $M_0$. For technical reasons, we ask these measures to assign measure zero to a fixed neighborhood of $0$ (which we will take, for convenience, to $[-1,1]$). Since our estimates are uniform on compact intervals, this choice will not affect our global convergence result, since this choice will only affect the various random holomorphic functions through the appearance of an extra factor $e^{c s}$, with an explicit constant $c$.

Last, note that in the next theorem we have an extra condition involving some function $\varphi$
tending to infinity at infinity arbitrarily slowly.  
The appearance of $\varphi$ may not seem natural, but it is there to ensure some uniform integrability condition. In all the examples presented in the article, the assumptions of the theorem should be checked, and the difficulty of this task does not increase by the presence of this function $\varphi$.

\begin{theorem} \label{main} 
Let $(X_n)_{n \in \mathbb{N}}$ and $X$ be simple point processes on the real line, locally finite, i.e. with finitely many points on compact sets. Let $M_{X_n}$ and $M_X$ be the corresponding random Radon measures, i.e. 
$$M_{X_n} := \sum_{ \rho \in X_n} \delta_{\rho}, \;  M_{X} := \sum_{ \rho \in X} \delta_{\rho},$$
where $\delta_{\rho}$ denotes the Dirac measure  at $\rho$. 
We assume that $X_n$ converges in law to $X$, in the sense that for all functions $g$ from $\mathbb{R}$ to $\mathbb{R}_+$,  continuous with compact support, 
$$\int_\mathbb{R} g \, dM_{X_n} \underset{n \rightarrow \infty}{\longrightarrow} \int_{\mathbb{R}} g \, dM_X$$
in distribution. Moreover, we suppose that there exist deterministic Radon measures $M_{0,n}$ and $M_0$ on the real line satisfying the following assumptions, for some 
 nondecreasing function $\varphi$ from $[0, \infty)$ to $(0, \infty)$, tending to infinity at infinity: 
\begin{itemize}
 \item We have $$\int_{\mathbb{R}} \frac{|M_0 (x)|}{|x|^3} dx < \infty,$$
and almost surely, 
 $$\int_{\mathbb{R}} \frac{|M_0 (x) - M_X(x) |}{x^2} dx < \infty, \; \frac{|M_0 (x) - M_X(x) |}{|x|} \underset{|x| \rightarrow \infty}{\longrightarrow} 0$$
\item When $n$ goes to infinity, the measure $M_{0,n}$ converges vaguely to $M_0$, i.e. for all 
functions $g$, continuous with compact support, 
$$\int_\mathbb{R} g dM_{0,n} \underset{n \rightarrow \infty}{\longrightarrow} \int_{\mathbb{R}} g dM_0.$$
\item The sequence
$$\left(\int_{\mathbb{R}} \frac{|M_{0,n}(x)| \varphi(|x|)}{|x|^3} dx \right)_{n \in \mathbb{N}},
$$ is bounded and the families of variables $$
\left(\int_{\mathbb{R}} \frac{|M_{0,n} (x) - M_{X_n}(x) | \varphi(|x|)}{x^2} dx \right)_{n \in \mathbb{N}}$$
and $$\left(\sup_{x \in \mathbb{R}} \frac{|M_{0,n} (x) - M_{X_n}(x) |\varphi(|x|)}{|x|} \right)_{n \in \mathbb{N}}$$
are tight.
\end{itemize}
Then, for all $n \in \mathbb{N}$, the random function 
$$s  \mapsto \lim_{A \rightarrow \infty} 
\exp \left( s \int_{[-A,A]} \frac{d M_{0,n}(\rho)}{\rho} \right)\prod_{\rho \in X_n\cap [-A,A]} \left(1- \frac{s}{\rho} \right)$$
is almost surely well-defined, entire, its zeros are simple and given by the point process $X_n$. Moreover, for $n \rightarrow \infty$, as a random element of  
the space of continuous functions from $\mathbb{C}$ to $\mathbb{C}$ endowed with the topology of the uniform convergence on compact sets,  it converges in distribution
to the
 function
 $$s  \mapsto \lim_{A \rightarrow \infty} 
\exp \left( s \int_{[-A,A]} \frac{d M_{0}(\rho)}{\rho} \right)\prod_{\rho \in X \cap [-A,A]} \left(1- \frac{s}{\rho} \right),$$
 which is almost surely well-defined, entire, and whose zeros are given by $X$.
\end{theorem}
\begin{remark}

The convergence in law, for the topology of uniform convergence on compact sets, is equivalent to the convergence of the finite-dimensional marginals, together with the tightness
of the maximal modulus of the function on any given compact set, as soon as the functions that are involved are holomorphic. This can be proven similarly as in \cite{CHNNR19}, using the notion of a compact-equicontinuous family of functions in probability, and the
fact that the supremum of the derivative of a holomorphic function on a given compact set is dominated by the supremum of the function itself on a slightly larger compact set. 

Moreover, the convergence involved in the theorem implies the convergence in law of all the ratios of products of values of the function taken at different points, as soon as all the
denominators are almost surely non-zero under the limiting probability distribution, which for example occurs when we take values of the functions outside the real line. If we have suitable properties of domination, we can also deduce the convergence of moments of ratios.
\end{remark}

A significant example can be constructed from the eigenvalues of the GUE. 
If the GUE of order $n$ is normalized in such a way that its eigenvalues $\lambda^{(n)}_1, \dots, \lambda^{(n)}_n$ have joint density proportional to 
$$e^{- \sum_{1 \leq j \leq n} (\lambda^{(n)}_j)^2/2} \, \prod_{1 \leq j < k \leq n} (\lambda^{(n)}_j - \lambda^{(n)}_k)^{2},$$
then 
the empirical measure associated to the points $\lambda^{(n)}_1/ \sqrt{n}, \lambda^{(n)}_2/\sqrt{n}, \dots, \lambda^{(n)}_n/\sqrt{n}$ tends in law to the semi-circular distribution, 
which is supported by the interval $[-2, 2]$. 
We now rescale the GUE around a point in the bulk of the spectrum: we fix $E \in (-2,2) $, and we 
define the point process $X_n$ by 
$$X_n := \left\{ \frac{ (\lambda^{(n)}_j - E \sqrt{n}) \sqrt{n (4 - E^2)}}{2 \pi}, 1 \leq j \leq n\right\}. $$
This point process converges in law to $X$, where $X$ is a determinantal sine-kernel point process. We choose the measures $M_{0,n}$ and $M_0$ as follows: 
 \begin{itemize}
 \item The measure $M_{0,n}$ is the restriction to $\mathbb{R} \backslash [-1,1]$ of the suitably rescaled semi-circular distribution: more precisely, $M_{0,n}$ has 
 a density
 $$x \mapsto (4 - E^2)^{-1/2} \sqrt{\max(0,4 - (E  + 2 \pi x n^{-1} (4 - E^2)^{-1/2})^2)} \mathds{1}_{x \not\in [-1,1]}.$$
 \item The measure $M_0$ is the Lebesgue measure on $\mathbb{R} \backslash [-1,1]$. 
 \end{itemize}
 These measures are chosen in such a way that they approximate the point measures corresponding to the point processes $X_n$ and $X$. 
 We will later prove that with this choice, the assumptions of Theorem \ref{main} are satisfied, and that the conclusion of this theorem 
 implies, after elementary computation, that
 $$ \frac{Z_{n}(E \sqrt{n} + 2 \pi s (n(4 - E^2))^{-1/2})}{Z_n (E \sqrt{n})}
\underset{n \rightarrow \infty}{\longrightarrow} \underset{A \rightarrow \infty}{\lim} e^{\frac{s \pi E}{\sqrt{4 - E^2}}} \, \prod_{\rho \in X \cap [-A,A]} \left(1 - \frac{s} {\rho} \right)$$
in distribution, for the topology of the uniform convergence on compact sets, where 
$$Z_n(z) := \prod_{j=1}^n (z - \lambda^{(n)}_j)$$
is the characteristic polynomial of a GUE matrix of order $n$. 
 This result has been proven in \cite{CHNNR19} in a different way.

In all the examples considered in this article, the assumptions of Theorem \ref{main} are checked by using the following result: 
\begin{proposition} \label{crit}
With the notation of Theorem \ref{main}, in the case where $(X_n)_{n \in \mathbb{N}}$ and $X$ have almost surely no point at zero, the assumptions on $M_{0,n}$ and $M_0$ are satisfied as soon as we have the following properties, for some $C > 0$, $\nu \in (0,1)$, $\alpha \in (0,1)$ such that $\nu + \alpha < 1$: 
\begin{itemize}
\item $M_{0,n}$ converges vaguely to $M_0$ when $n$ goes to infinity. 
\item For all $n \in \mathbb{N}$, $M_{0,n}( [-1,1]) = 0$.
\item For all $n \in \mathbb{N}$, $x, y \in \mathbb{R}$,
 $$|M_{0,n} ([x,y])|  \leq C \, (1 + |y-x|)\, (|x|^{\nu} + |y|^{\nu}),$$
  $$|M_{0} ([x,y])|  \leq C \, (1 + |y-x|)\, (|x|^{\nu} + |y|^{\nu}).$$
 \item For all $n \in \mathbb{N}$, $x \in \mathbb{R}$,
 $$| M_{0,n} (x) - \mathbb{E} [ M_{X_n}(x) ]| \leq C (1 + |x|^{\alpha}),$$
 $$ | M_{0} (x) - \mathbb{E} [ M_{X}(x) ]| \leq C (1 + |x|^{\alpha}).$$
\item For all $n \in \mathbb{N}$, $x \in \mathbb{R}$,
$$\operatorname{Var}(M_{X_n}(x)) \leq C (1 +  |x|^{2\alpha}),$$
$$\operatorname{Var}(M_{X}(x)) \leq C (1 +  |x|^{2 \alpha}).$$
\end{itemize}
\end{proposition}

\begin{remark}
 If the point processes $(X_n)_{n \in \mathbb{N}}$ are uniformly product-amenable in the sense of \cite{CHNNR19}, and if we take $M_{0,n}(x) = \mathbb{E} [ M_{X_n}(x)]$, 
then the conditions of Proposition \ref{crit} relative to $X_n$ are satisfied, for $\nu$ small and $\alpha$ sufficiently close to $1$. Moreover, the assumptions of Proposition \ref{crit} imply 
that 
$$\int_{\mathbb{R}} \frac{d M_{X_n}(x)}{ 1+ x^2} < \infty, \; \int_{\mathbb{R}} \frac{d M_{X}(x)}{ 1+ x^2} < \infty$$
almost surely (see proof of Proposition~\ref{weierstrass}). Because of the above integrability conditions, one could imagine that there exist random operators $\{\tau_n\}_{n\geq 1}$ and $\tau_\infty$, with spectrum the points of $\{X_n\}$ and $X$ respectively, such that the operators $\tau_n$ converge to $\tau$ in the right norm (this is in the spirit of the work \cite{VV20}). But this is beyond the scope of this paper, and we hope to be able to address this problem in a future work.

\end{remark}

A converse of Theorem \ref{main}, less difficult to prove and using weaker assumptions, can be stated as follows:
\begin{theorem} \label{converse}
Let $(f_n)_{n \in \mathbb{N}}$ be a family of random holomorphic functions on an open set $D$ of $\mathbb{C}$. We assume that $(f_n)_{n \in \mathbb{N}}$, 
as a sequence of random elements of the space $\mathcal{C}(D, \mathbb{C})$ endowed with the topology of uniform convergence on compact subsets of $D$, 
 converges in distribution to a random function $f$, which is almost surely holomorphic in $D$. 
We assume that there is almost surely no nonempty open ball in $D$ where $f$, or $f_n$ for some $n \in \mathbb{N}$, is identically zero, implying that the
zeros of $f_n$ and $f$ are isolated. Then, the point process of the zeros of $f_n$ converges in law to the point process of the 
zeros of $f$, in the following sense: for all continuous function $g$ from $D$ to $\mathbb{R}_+$, with compact support, 
$$\sum_{\rho \in D, f_n(\rho) = 0} \mu_{f_n} (\rho) g(\rho) \underset{n \rightarrow \infty}{\longrightarrow} \sum_{\rho \in D, f(\rho) = 0} \mu_{f} (\rho)  g(\rho)$$
in distribution, where $\mu_{f_n}(\rho)$ or $\mu_f(\rho)$ denote the multiplicities of $\rho$ as a zero of $f_n$ or $f$. 

\end{theorem}
This theorem can be related to the results proven by Shirai in \cite{S12}.
 Notice that in Theorem \ref{converse}, we do not assume that the holomorphic functions are whole, and we use very weak assumptions on their zeros. 
 
 \section{Logarithmic derivatives of Stieltjes transforms of random measures}

Since the derivative of an analytic function can be written as a suitable contour integral of the function itself, and since all the zeros of the holomorphic functions involved in this article are
on the real line, one deduces the following corollary, which can be seen as a probabilistic version of convergence results on Stieltjes transforms proven by Aizenman and Warzel: see \cite{AW15}, Theorem 4.1, for example.  
\color{black}
\begin{corollary} \label{coro}
Under the assumptions of Theorem \ref{main}, the logarithmic derivative of 
$$s  \mapsto \lim_{A \rightarrow \infty} 
\exp \left( s \int_{[-A,A]} \frac{d M_{0,n}(\rho)}{\rho} \right)\prod_{\rho \in X_n\cap [-A,A]} \left(1- \frac{s}{\rho} \right),$$
which is well-defined for $s \in \mathbb{C} \backslash \mathbb{R}$, can be written as 
$$s  \mapsto \lim_{A \rightarrow \infty} 
\left( \int_{[-A,A]} \frac{d M_{0,n}(\rho)}{\rho}  + \sum_{\rho \in X_n\cap [-A,A]} \frac{1}{s-\rho} \right),$$
and tends in law to 
the logarithmic derivative of 
$$s  \mapsto \lim_{A \rightarrow \infty} 
\exp \left( s \int_{[-A,A]} \frac{d M_{0}(\rho)}{\rho} \right)\prod_{\rho \in X \cap [-A,A]} \left(1- \frac{s}{\rho} \right),$$
equal to 
$$s  \mapsto \lim_{A \rightarrow \infty} 
 \left(\int_{[-A,A]} \frac{d M_{0}(\rho)}{\rho}  + \sum_{\rho \in X \cap [-A,A]} \frac{1}{s-\rho} \right),$$
 for the topology of uniform convergence on compact sets of $\mathbb{C} \backslash \mathbb{R}$. 
\end{corollary}

In \cite{AW15}, some universal properties of the Cauchy distribution are proven, for the value of a large class of Herglotz functions, i.e. functions mapping the complex upper half-plane to itself, which includes some Stieltjes transform of measures. In this article, we prove universal properties of the Cauchy distribution in a probabilistic setting when the distribution of the underlying measure is translation-invariant. 
\begin{theorem} \label{Cauchy}
Let $M$ be a random atomic measure on $\mathbb{R}$, with finitely many atoms on compact sets. We assume that the distribution of $M$ is invariant by translation, and that for some $C >0$, $\alpha \in [0,1)$, 
$$ 0 <  K:= \mathbb{E}  [M(1) ] < \infty$$
and for all $x > 0$, 
$$\operatorname{Var} ( M(x)) \leq C (1+ |x|^{2 \alpha} ).$$
Then, for all $s \in \mathbb{R}$, the Stieltjes transform of $M$ at $s$, given by 
$$\frac{1}{K \pi} \underset{A \rightarrow \infty}{\lim} \int_{-A}^A \frac{dM(\rho)}{s-\rho} $$
is almost surely well-defined, and follows the standard Cauchy distribution, with density $$x \mapsto \frac{1}{\pi (1 + x^2)}$$ with respect to the Lebesgue measure. 
\end{theorem}

Note that in the work of Aizenman and Warzel, one needs the shift-amenability assumption of the function which is considered, i.e. the existence of a limiting measure for the image by the function of the uniform probability distribution on $[-L,L]$, when $L$ tends to infinity. In the present setting, the assumption of shift-amenability of the Stieltjes transform of $M$ is not made, and it is replaced by the invariance by translation of the distribution of $M$. Our result applies, for instance, to the measure $M$ associated with the Sine$_{\beta}$ point process. But to our knowledge, it is not known whether the Stieltjes transform of the Sine$_{\beta}$ point process is shift-amenable or not.

Another example of a setting where Theorem \ref{Cauchy} applies is discussed in detail by Najnudel and Vir'ag in \cite{NV21bead}. In this setting, one takes for $M$ a measure with support given by a Sine$_{\beta}$ point process, normalized in such a way that the average spacing between the points is $ 2\pi$, the measure of points being i.i.d. Gamma variables of shape parameter $\beta/2$ and mean $2$. In this case, $K = 1/\pi$, and then 
$$\underset{A \rightarrow \infty}{\lim} 
\int_{-A}^A \frac{d M (\rho)}{s - \rho}$$
is a standard Cauchy random variable.

We also give a full proof of the following result on the Riemann zeta function, which was first stated by Sodin (\cite{S18}) who sketched a proof of it and who shared his ideas of the proof with us:
\begin{theorem}[Sodin \cite{S18}] \label{Cauchy-Riemann}
Let us assume the Riemann hypothesis, and  
let $\omega$ be a uniform random variable on $[0,1]$. Then, the family of random variables
$$ \left(\frac{2 i }{\log T} \frac{\zeta'}{\zeta} \left( \frac{1}{2} + i T \omega \right)  +  i \right)_{T \geq 2}$$
converges in distribution to a standard Cauchy variable, when $T$ tends to infinity. 
\end{theorem}

 Theorem \ref{main} and Corollary \ref{coro} are proven in Section \ref{proof-theorem}. 
  Proposition \ref{crit}  and Theorem \ref{converse} are proved respectively in Sections \ref{criterion} and \ref{proof-converse}. 
    In Section \ref{examples}, we provide a number of significant examples for which Theorem \ref{main} applies. 
In Section \ref{cauchy}, we prove Theorems \ref{Cauchy} and \ref{Cauchy-Riemann}. 

\section*{Acknowledgements}

We express our deepest thanks to Sasha Sodin for taking the time to explain to us some of the arguments in his paper \cite{S18}. More specifically he explained to us how one could recover the stochastic zeta function from the results of Aizenman and Warzel (\cite{AW15}) on the  Stieltjes transform of the Sine kernel by careful integration of the logarithmic derivative, using our uniform estimates on the expected number of points in compact intervals and their variance. He also very generously shared with us his ideas for the proof of Theorem~\ref{Cauchy-Riemann}.

\section{Proof of the Theorem \ref{main} and Corollary \ref{coro} } \label{proof-theorem}

The proof of Theorem \ref{main} is obtained by combining several intermediate results, which are successively stated and proven in this section. 
The first result provides sufficient conditions under which a convergence in law of random holomorphic functions can be deduced from a convergence of the 
finite-dimensional marginals, which is more convenient to control.  
\begin{proposition} \label{finite}
Let $(f_n)_{n \in \mathbb{N}}$ be a sequence of random entire functions. We assume that the finite dimensional marginals of $f_n$ converge in distribution to the marginals of a random function $f$, and that
for all compacts $K$, the family of random variables $(\sup_K |f_n|)_{n \in \mathbb{N}}$ is tight. Then $f$ is entire and $f_n$ converges in law to $f$ for the topology of uniform convergence on compact sets. 
\end{proposition}
\begin{proof}
We have 
$$f'_n(s) = \frac{1}{2 \pi} \int_{0}^{2 \pi} f_n(s + e^{i \theta})  e^{- i \theta} d\theta,$$
which shows that for all compacts $K$, $(\sup_K |f'_n|)_{n \in \mathbb{N}}$ is also tight. 
For $ \delta > 0$, and $\Delta > 0$, we have 
$$\underset{n \rightarrow \infty}{\lim\sup} \, \mathbb{P} \left( \sup_{|s_1 - s_2| < \Delta, s_1, s_2 \in K} |f_n(s_1) - f_n(s_2)| \geq \delta \right)
$$ $$\leq \underset{n \rightarrow \infty}{\lim\sup} \, \mathbb{P} \left( \sup_K |f'_n| \geq \delta/\Delta \right)
\underset{\Delta \rightarrow 0}{\longrightarrow} 0,$$
by the tightness of $(\sup_K |f'_n|)_{n \in \mathbb{N}}$. We deduce that for $ \delta > 0$, $\varepsilon > 0$, the upper limit above 
is smaller than $\varepsilon$ if $\Delta$ is small enough. From Definition 2.4. and Lemma 2.5 in \cite{CHNNR19}, we deduce that 
$f_n$ converges to $f$ for the topology of uniform convergence on compact sets, which then implies that $f$ is entire. 

\end{proof}

In Theorem \ref{main}, we remark that the product of $(1- s/\rho)$ is restricted to an interval $[-A,A]$, since in most interesting examples, this product is not 
absolutely convergent. A double limit $A \rightarrow \infty$ and $n \rightarrow \infty$ in then involved. 
In order to be able to handle this double limit, the following result is useful: 
\begin{proposition} \label{fnA}
Let $(f_n)_{n \in \mathbb{N}}$, $(f_{n,A})_{n \in \mathbb{N}, A > 0}$, $f_{\infty}$, $(f_{\infty,A})_{A > 0}$, be random entire functions, defined on the same probability space for each $n$ (but not necessarily for all $n$), satisfying the following assumptions: 
\begin{itemize}
\item Almost surely, for all $n \in \mathbb{N}$ and $A > 0$, $f_{n,A}$ is not identically zero. 
\item Almost surely, for all $n \in \mathbb{N}$ integer or $n = \infty$, $f_{n,A}$ converges to $f_n$, uniformly on compact sets, when $A$ goes to infinity. 
\item For all $A > 0$, and all $K$ compact, $(\sup_{K} |f_{n,A}|)_{n \in \mathbb{N}}$ is tight. 
\item We have, for all $\varepsilon > 0$ and $K$ compact, 
$$\sup_{n \in \mathbb{N}} \mathbb{P} \left( \sup_{K} \left( |(f_{n}/f_{n,A}) - 1| \mathds{1}_{f_{n,A} \neq 0} \geq \varepsilon \right) \right)
\underset{A \rightarrow \infty}{\longrightarrow} 0.$$
\item For all $A > 0$, the finite dimensional marginals of $f_{n,A}$ converge to those of $f_{\infty, A}$ in distribution. 
\end{itemize}
Then, $f_n$ converges in distribution to $f_{\infty}$ when $n$ goes to infinity, for the topology of uniform convergence on compact sets.
\end{proposition}
\begin{proof}
Let us first show the tightness of $(\sup_{K} |f_n|)_{n \in \mathbb{N}}$ for a given compact set $K$. It is equivalent to consider a general open disc $D$ instead of $K$. 
For fixed $\delta > 0$, we have  
$$\mathbb{P} \left( \sup_{\overline{D}} \left( |(f_{n}/f_{n,A}) - 1| \mathds{1}_{f_{n,A} \neq 0} \geq 0.01 \right) \right) \leq \delta$$
if $A$ is large enough, independently of $n$. For such $A$, we have, with probability $\geq 1- \delta$, 
$$\sup_{D \backslash \{s \in D, f_{n,A}(s) = 0\}} |f_n| \leq 1.01 \sup_{D} |f_{n,A}|.$$
Since $f_n$ and $f_{n,A}$ are assumed to be entire, $f_n$ is continuous, and the zeros of $f_{n,A}$ are isolated, since $f_{n,A}$ is not identically zero by assumption.
Hence, 
$$\sup_{D} |f_n| \leq 1.01 \sup_{D} |f_{n,A}|$$
with probability  $\geq 1- \delta$. 
We deduce, for $A$ large enough depending on $\delta$, and $B > 0$, 
$$\mathbb{P} \left( \sup_{D} |f_n|  \geq B \right) \leq  \delta + \mathbb{P} \left( \sup_{D} |f_{n,A}| \geq 0.99 B \right).$$
By tightness of $(\sup_{D} |f_{n,A}|)_{n \in \mathbb{N}}$, the last probability is smaller than $\delta$, independently of $n$, for $B$ large enough depending on $A$, and then for $B$
large enough depending on $\delta$. This gives, independently of $n$,  
$$\mathbb{P} \left( \sup_{D} |f_n|  \geq B \right) \leq 2 \delta$$
for $B$ large enough depending on $\delta$, i.e. the tightness of  $(\sup_{D} |f_n|)_{n \in \mathbb{N}}$. 

We also have, for $\varepsilon, \eta \in (0,1/2)$, since the zeros of $f_{n,A}$ are isolated, 
 \begin{align*}
\mathbb{P} & \left( \sup_{D} |f_n - f_{n,A}| >  \varepsilon \right)
 = \mathbb{P} \left( \sup_{D, f_{n,A}(s) \neq 0} |f_n - f_{n,A}| > \varepsilon \right)
\\ & \leq \mathbb{P} \left( \sup_{D, f_{n,A}(s) \neq 0} |(f_n/f_{n,A}) - 1| \geq \varepsilon \eta \right)
+  \mathbb{P} \left( \sup_{D} |f_n| \geq 1/2\eta \right)
\end{align*}
If we take the supremum in $n$ and then the upper limit in $A$, the first term goes to zero by assumption. 
By tightness of $(\sup_{D} |f_n|)_{n \in \mathbb{N}}$, the second term is bounded by a quantity going to zero with $\eta$. 
Hence, 
$$\sup_{n \in \mathbb{N}} \mathbb{P}  \left( \sup_{D} |f_n - f_{n,A}| > \varepsilon \right) \underset{A \rightarrow \infty}{\longrightarrow} 0.$$

From the tightness of $(\sup_{K} |f_n|)_{n \in \mathbb{N}}$ for $K$ compact, and from Proposition \ref{finite}, it is now enough to prove the convergence of the finite-dimensional marginals. We consider a functional $F$ on the space of continuous functions from $\mathbb{C}$ to $\mathbb{C}$, which depends only on the value of the functions at finitely many points $s_1, \dots, s_p$, which is bounded by a constant $M > 0$, and which is $1$-Lipschitz with respect to each of the values at $s_1, \dots, s_p$. It is enough to show that 
$$\mathbb{E} [ F(f_n)] \underset{n \rightarrow \infty}{\longrightarrow} \mathbb{E} [F (f_{\infty})].$$
We have, for all $\varepsilon > 0$,
\begin{align*}|\mathbb{E} [F (f_{\infty})]  - \mathbb{E} [F (f_{n})]| &
\leq |\mathbb{E} [F (f_{\infty,A})] - \mathbb{E} [F (f_{n,A})]|
\\ & +  \mathbb{E} \left[ \min\left(2M, p \sup_{\{s_1, \dots, s_p\}} |f_{n} - f_{n,A}| \right) \right] 
\\ & +  \mathbb{E} \left[ \min\left(2M, p \sup_{\{s_1, \dots, s_p\}} |f_{\infty} - f_{\infty,A}| \right)\right]
\\ & \leq  |\mathbb{E} [F (f_{\infty,A})] - \mathbb{E} [F (f_{n,A})]|
\\ & + p \varepsilon + 2M \mathbb{P} \left(\sup_{\{s_1, \dots, s_p\}} |f_{n} - f_{n,A}| > \varepsilon \right)
\\ & + \mathbb{E} \left[ \min \left(2M, p \sup_{\{s_1, \dots, s_p\}} |f_{\infty} - f_{\infty,A}| \right)\right].
\end{align*}
Since the finite-dimensional marginals of $f_{n,A}$ tend to those of $f_{\infty, A}$ when $n$ goes to infinity, we have
\begin{align*}
\underset{n \rightarrow \infty}{\lim \sup} \, |\mathbb{E} [F (f_{\infty})]  - \mathbb{E} [F (f_{n})]|
& \leq p \varepsilon + 2M \sup_{n \in \mathbb{N}} \mathbb{P} \left(\sup_{\{s_1, \dots, s_p\}} |f_{n} - f_{n,A}| > \varepsilon \right)
\\ & +\mathbb{E} \left[ \min \left(2M, p \sup_{\{s_1, \dots, s_p\}} |f_{\infty} - f_{\infty,A}| \right)\right].
\end{align*}
We have proven that the second term tends to zero when $A$ goes to infinity. By assumption, $f_{\infty,A}$ tends a.s. to $f_{\infty}$ (uniformly on compact sets), so 
the last term also goes to zero. By letting $A \rightarrow \infty$, and then $\varepsilon \rightarrow 0$, we deduce that 
$$\underset{n \rightarrow \infty}{\lim \sup} \, |\mathbb{E} [F (f_{\infty})]  - \mathbb{E} [F (f_{n})]| = 0,$$
which gives the convergence of the finite-dimensional marginals of $f_n$ towards those of $f_{\infty}$, and completes the proof of the proposition. 
\end{proof}

In Theorem \ref{main}, in the case where $0 \notin X$,  the product 
$$\prod_{\rho \in X \in [-A,A]} \left(1- \frac{s}{\rho} \right)$$ 
provides a holomorphic function whose zeros are given by the set $X \cap [-A,A]$. 
In order to get a function whose zeros are given by $X$, we let $A \rightarrow \infty$. If the product 
$$\prod_{\rho \in X} \left(1- \frac{s}{\rho} \right)$$ 
is not absolutely convergent,  another way to solve this issue is to consider the  Weierstrass product
$$\prod_{\rho \in X} \left(1- \frac{s}{\rho} \right) e^{s/\rho}$$
if this last product is convergent. In the next proposition, we make the connection between the two points of view. 


\begin{proposition} \label{weierstrass}
Let $X$ a simple point process on the real line, locally finite.  Let $M_X$ be the random measure corresponding to $X$, and let $M_0$ be a deterministic Radon measure on the real line satisfying the following conditions, given in Theorem \ref{main}: 
$$\int_{\mathbb{R}} \frac{|M_0 (x)|}{|x|^3} dx < \infty$$ and almost surely, 
$$\int_{\mathbb{R}} \frac{|M_0 (x) - M_X(x) |}{x^2} dx < \infty, \; \frac{|M_0 (x) - M_X(x) |}{|x|} \underset{|x| \rightarrow \infty}{\longrightarrow} 0.$$
Then, almost surely, the function
$$s \mapsto \exp \left( s \int_{[-A,A]} \frac{d M_0(\rho)}{\rho} \right)\prod_{\rho \in X \cap [-A,A]} \left(1- \frac{s}{\rho} \right),$$
is well-defined and, for $A \rightarrow \infty$, it  converges uniformly on compact sets to the function
$$ s \mapsto \exp \left(  s \int_{\mathbb{R}}  \frac{ M_0(x) - M_X (x)}{x^2} dx \right) \prod_{\rho \in X} \left(1- \frac{s}{\rho} \right) e^{s/\rho},$$
the  Weierstrass product being absolutely convergent, uniformly on compact sets.
\end{proposition}
\begin{proof}
We have 
$$\sum_{\rho \in X} \rho^{-2} = \sum_{\rho \in X} \int_{\mathbb{R}} \frac{2 dx}{|x|^3} (\mathds{1}_{x \geq \rho \geq 0} + \mathds{1}_{x < \rho < 0} )
= \int_{\mathbb{R}} \frac{2 |M_X(x)|}{|x|^3} dx,$$
and then 
\begin{align*}
\sum_{\rho \in X} \rho^{-2} & \leq \int_{[-1,1]} \frac{2 |M_X(x)|}{|x|^3} dx
+ \int_{\mathbb{R} \backslash [-1,1]} \frac{2 |M_0(x)|}{|x|^3} dx
\\ & +  \int_{\mathbb{R} \backslash [-1,1]} \frac{2 |M_0(x) - M_X(x)|}{|x|^3} dx.
\end{align*}
The two last integrals are almost surely finite by assumption. The first integral is also almost surely finite, since $X$ is assumed to 
be locally finite, and has almost surely no point at zero, because for any $A > 0$ (note that since we are working on a compact interval it is enough to check the integrability condition with $x^2$ in the denominator) 
$$\int_{[-A,A]} \frac{|M_X(x) |}{x^2} dx
\leq A \int_{[-A,A]} \frac{|M_0(x) |}{|x|^3} dx + \int_{[-A,A]} \frac{|M_0 (x) - M_X(x) |}{x^2} dx  < \infty.$$
Hence, 
$$\sum_{\rho \in X} \rho^{-2} < \infty$$
almost surely, which ensures the pointwise convergence of the  Weierstrass product
$$s \mapsto \prod_{\rho \in X} \left(1- \frac{s}{\rho} \right) e^{s/\rho},$$
 and then its uniform convergence on compact sets by Montel's theorem. 
Now,  for a Radon measure $M$ on $\mathbb{R}$ and for $y \in \mathbb{R}$, 
let us denote 
$$M(y-) := \underset{y' \rightarrow y, \, y' < y}{\lim} M(y'),$$
and for $y, t \in \mathbb{R}$,  let us define:
$$\max(y-, t) := y-, \; \min(y-, t) := t$$
for $y > t$, and 
$$\max(y-, t) = t, \; \min(y-, t) = y-$$
for $y \leq t$. Then, we have, for $A > 0$,  

\begin{align*}
\int_{[-A,A]} \frac{d M_0(\rho)}{|\rho|}  
& = \int_{[-A,A]} d M_0(\rho) \int_{\mathbb{R}}  \frac{dx}{x^2} (\mathds{1}_{x \geq \rho \geq 0} + \mathds{1}_{x < \rho < 0} )
\\ & = \int_{\mathbb{R}} \frac{ |M_0 (\min(A,\max((-A)-,x)))|}{x^2} dx
\\ &  \leq A \int_{[-A,A]} \frac{|M_0 (x)|}{|x|^3} dx + \left(\frac{|M_0(A)| + |M_0((-A)-)|}{A} \right)< \infty.
\end{align*}
Since the process $X$ is locally finite, this bound ensures that the function depending on $A$ in the proposition is well-defined. Moreover, we have: 
\begin{align*}
\int_{[-A,A]} \frac{d M_0(\rho)}{\rho}  
& = \int_{[-A,A]} d M_0(\rho) \int_{\mathbb{R}}  \frac{dx}{x^2} (\mathds{1}_{x \geq \rho \geq 0} - \mathds{1}_{x < \rho < 0} )
\\ & = \int_{\mathbb{R}} \frac{ M_0 (\min(A,\max((-A)-,x)))}{x^2} dx,
\end{align*}
and a similar equality for $d M_X$, which gives 
$$\int_{[-A,A]} \frac{d M_0(\rho)}{\rho} - \int_{[-A,A]} \frac{d M_X(\rho)}{\rho} $$ $$
=  \int_{\mathbb{R}} \frac{ M_0 (\min(A,\max((-A)-,x))) - M_X (\min(A,\max((-A)-,x)))}{x^2} dx.$$

The function depending on $A$ in the proposition is then equal to 
$$\exp \left(s \int_{\mathbb{R}} \frac{ M_0 (\min(A,\max((-A)-,x))) - M_X (\min(A,\max((-A)-,x)))}{x^2} dx \right) \dots
$$ $$ \times \prod_{\rho \in X \cap [-A,A]} \left(1- \frac{s}{\rho}  \right) e^{s / \rho}.$$
In the exponential factor, the part of the integral in the interval $[-A,A]$ is
$$\int_{[-A,A]} \frac{M_0(x)-M_X(x)}{x^2} dx \underset{A \rightarrow \infty}{\longrightarrow} \int_{\mathbb{R}} \frac{M_0(x)-M_X(x)}{x^2} dx.$$
by dominated convergence. 
The part of the integral outside $[-A,A]$ is equal to 
$$\frac{M_0(A) - M_X(A)}{A} + \frac{M_0((-A)-) - M_X((-A)-)}{A}.$$
The first term goes to zero when $A \rightarrow \infty$. It is also the case for the second term, since it is 
the left limit at $-A$ of the c\`adl\`ag function
$$y \mapsto \frac{M_0(-y) - M_X(-y)}{y}$$
from $(-\infty, 0)$ to $\mathbb{R}$, 
which tends to zero at $-\infty$ by assumption. 
This completes the proof of the proposition, by the convergence of the   Weierstrass  product. 
\end{proof}

We have now enough material to prove Theorem \ref{main}. 
The almost sure existence of the limits
$$\lim_{A \rightarrow \infty} 
\exp \left( s \int_{[-A,A]} \frac{d M_{0,n}(\rho)}{\rho} \right)\prod_{\rho \in X_n\cap [-A,A]} \left(1- \frac{s}{\rho} \right)$$
and 
 $$ \lim_{A \rightarrow \infty} 
\exp \left( s \int_{[-A,A]} \frac{d M_{0}(\rho)}{\rho} \right)\prod_{\rho \in X \cap [-A,A]} \left(1- \frac{s}{\rho} \right)$$
 and the fact that they are entire functions of $s$, come from the assumptions of Proposition \ref{weierstrass}, which are clearly satisfied for 
each of the point processes $(X_n)_{n \in \mathbb{N}}$, $X$ and the corresponding deterministic measures $(M_{0,n})_{n \in \mathbb{N}}$, $M_0$. 
Since $M_X$ is Radon by assumption, for all $A > 0$, $\varepsilon > 0$, 
there exists $L > 0$ such that $M_X([-A,A]) \leq L$ with probability at least $1 - (\varepsilon/2)$. 
If there exist $m$ points in $[-A,A]$ which are in $X$ with probability at least $\varepsilon$, we get for each of these points $p$: 
$$\mathbb{P} [  p \in X, M_X([-A,A]) \leq L] \geq \mathbb{P} [p \in X] - \mathbb{P} [ M_X([-A,A]) > L] \geq \varepsilon/2,$$
and then 
$$L \geq \mathbb{E} [ M_X([-A,A]) \mathds{1}_{M_X([-A,A]) \leq L} ] \geq m \varepsilon/2,$$
which shows that $m$ is finite. Letting $\varepsilon \rightarrow 0$ and $A \rightarrow \infty$, we deduce that there exist at most countably many points 
which are in $X$ with positive probability. Moreover, $M_0$ has also countably many atoms. A similar reasoning shows the same thing for $X_n$ and $M_{0,n}$. 
Hence, for all $A > 0$, we can choose $A' \in (A,A+1)$ such that $A' \in X$ and $-A' \in X$ with probability zero, for all $n \in \mathbb{N}$, $A' \in X_n$ and  $-A' \in X_n$ with probability zero, and $-A', A'$ are not atoms of $M_0$ or $M_{0,n}$. 
It is enough to show that one can apply Proposition \ref{fnA} to the functions 
$$f_{n,A}(s)  = \exp \left( s \int_{[-A',A']} \frac{d M_{0,n}(\rho)}{\rho} \right)\prod_{\rho \in X_n\cap [-A',A']} \left(1- \frac{s}{\rho} \right),$$
$$f_{\infty,A}(s) = \exp \left( s \int_{[-A',A']} \frac{d M_{0}(\rho)}{\rho} \right)\prod_{\rho \in X \cap [-A',A']} \left(1- \frac{s}{\rho} \right),$$
the functions $f_n$ and $f_{\infty}$ being their corresponding uniform limits on compact sets, which are proven to exist. 
We have $f_{n,A} (0) = 1$, so $f_{n,A}$ is not identically zero. 
An earlier computation gives that 
$$f_{n,A}(s) = \exp \left(s \int_{\mathbb{R}} \frac{ M_{0,n} (\min(A',\max((-A')-,x))) - M_{X_n} (\min(A',\max((-A')-,x)))}{x^2} dx \right) \dots
$$ $$ \times \prod_{\rho \in X_n \cap [-A',A']} \left(1- \frac{s}{\rho}  \right) e^{s / \rho}.$$
The integral in the  exponential factor is 
$$\int_{[-A',A']} \frac{M_{0,n}(x)-M_{X_n}(x)}{x^2} dx 
+ \frac{M_{0,n}(A') - M_{X_n}(A')}{A'} $$ $$+ \frac{M_{0,n}((-A')-) -M_{X_n}((-A')-)}{A'}.$$
which, by assumptions of Theorem \ref{main}, is tight for fixed $A$ and varying $n$: notice that  
we can replace $(-A')-$ by $-A'$ in the last term, because $-A'$ is almost surely not an atom of $M_{0,n}$ or $M_{X_n}$. 
On the other hand, we have, for $|x| \leq 1/2$, 
$$|(1-x) e^x| = \left| \exp\left( - \frac{x^2}{2} - \frac{x^3}{3} - \dots \right) \right| \leq e^{|x|^2}$$
and for $|x| \geq 1/2$, 
$$|(1-x) e^x| \leq (1 + |x|) e^{|x|} \leq e^{2 |x|} \leq e^{4 |x|^2}.$$
Hence, the  Weierstrass product is bounded by 
$$\exp \left( 4 |s|^2 \sum_{\rho \in X_n \cap [-A',A']} \rho^{-2}\right).$$
Let us prove that the last sum is tight with respect to $n$. We have for $L > 0$, 
\begin{align*}
\mathbb{P} \left[  \sum_{\rho \in X_n \cap [-A',A']} \rho^{-2} \geq L^3\right]
& \leq \mathbb{P} \left( \operatorname{Card} (X_n \cap [-1/L, 1/L]  ) \geq 1 \right)
\\ &  + \mathbb{P} \left(  \operatorname{Card} (X_n \cap [-A', A']  ) \geq L \right).
\end{align*}
Since $X_n$ converges to $X$, the upper limit in $n$ of the probability is at most 
$$\mathbb{P} \left( \operatorname{Card} (X \cap [-2/L, 2/L]  ) >0 \right)
  + \mathbb{P} \left(  \operatorname{Card} (X \cap [-2A', 2A'] )> L-1 \right).$$
  Since $X$ has a.s. no point at zero and is locally finite, the two probabilities go to zero when $L$ goes to infinity. Hence, we have the tightness of the sum 
of $\rho^{-2}$, and then the tightness of 
 $(\sup_{K} |f_{n,A}|)_{n \in \mathbb{N}}$ for all compacts sets $K$. 
 
 We also have (when the denominator is non-zero), by using Proposition \ref{weierstrass} and some computations above:
 $$\frac{f_n(s)}{f_{n,A}(s)}
 = \exp\left(s \left( \int_{\mathbb{R} \backslash [-A',A']} \frac{M_{0,n}(x)-M_{X_n}(x)}{x^2} dx 
-  \frac{M_{0,n}(A') - M_{X_n}(A')}{A'}  \right. \right. $$ $$ \left. \left. - \frac{M_{0,n}((-A')-) -M_{X_n}((-A')-)}{A'} \right) \right)
\prod_{\rho \in X_n \cap (\mathbb{R} \backslash [-A',A'])} \left(1- \frac{s}{\rho}  \right) e^{s / \rho}$$
 For $\varepsilon > 0$, and $K$ compact, we can only have $|(f_n/f_{n,A}) - 1| \mathds{1}_{f_{n,A} \neq 0} \geq \varepsilon$ on some point of $K$ if one of the following quantities 
 is larger than $\delta > 0$, depending only on $K$ and $\varepsilon$: 
 $$\int_{\mathbb{R} \backslash [-A',A']} \frac{|M_{0,n}(x)-M_{X_n}(x)|}{x^2} dx, \; 
  \frac{|M_{0,n}(A') - M_{X_n}(A')|}{A'},$$ $$
   \frac{|M_{0,n}((-A')-) -M_{X_n}((-A')-)|}{A'}, \;  \sum_{\rho \in X_n \cap (\mathbb{R} \backslash [-A',A'])} \rho^{-2}.$$
 For the first quantity, we need
 $$\int_{\mathbb{R}} \frac{|M_{0,n}(x)-M_{X_n}(x)| \varphi(|x|)}{x^2} dx \geq \delta \varphi(A').$$
 For the second quantity, we need
 $$ \frac{|M_{0,n}(A') - M_{X_n}(A')|  \varphi(|A'|)}{A'} \geq \delta \varphi(A').$$
 For the third quantity, we need
 $$ \frac{|M_{0,n}(-A') - M_{X_n}(-A')|  \varphi(|A'|)}{A'} \geq \delta \varphi(A'),$$
 notice that we can replace $(-A')-$ by $-A'$ since $M_{0,n}$ and $X_n$ have a.s. no atom at $-A'$. 
 For the fourth quantity, we need
 $$ \int_{\mathbb{R} \backslash [-A',A']} \frac{|M_{X_n}(x)| }{|x|^3} dx \geq \delta/2$$
 and then 
 $$ \int_{\mathbb{R} \backslash [-A',A']} \frac{|M_{0,n}(x)| }{|x|^3} dx \geq \delta/4$$
 or 
 $$ \int_{\mathbb{R} \backslash [-A',A']} \frac{|M_{0,n}(x) - M_{X_n}(x)| }{x^2} dx \geq
 A'  \int_{\mathbb{R} \backslash [-A',A']} \frac{|M_{0,n}(x) - M_{X_n}(x)| }{|x|^3} dx \geq A' \delta/4,$$
 which implies
  $$\int_{\mathbb{R}} \frac{|M_{0,n}(x)| \varphi(|x|)}{|x|^3} \geq \delta \varphi(A')/4$$
  or 
  $$\int_{\mathbb{R} \backslash [-A',A']} \frac{|M_{0,n}(x) - M_{X_n}(x)| \varphi(|x|)}{x^2} \geq A' \delta \varphi(A')/4.$$
 The tightness assumptions and the fact that $\varphi(A')$ goes to infinity with $A$ imply that the probability that $|(f_n/f_{n,A}) - 1| \mathds{1}_{f_{n,A} \neq 0} \geq \varepsilon$ 
 on some point in $K$ tends to zero when $A$ goes to infinity, uniformly in $n$. 
 In order to apply Proposition \ref{fnA}, it remains to show that the finite-dimensional marginals of $f_{n,A}$ converge in law to those of $f_{\infty,A}$. 
 We first have: 
 $$\int_{[-A',A']} \frac{d M_{0,n} (\rho)}{\rho} \underset{n \rightarrow \infty}{\longrightarrow} 
 \int_{[-A',A']} \frac{d M_0 (\rho)}{\rho}.$$
 Indeed, the vague convergence of $M_{0,n}$ towards $M_0$ implies that for all $\varepsilon > 0$, 
 $$\int_{[-A',A']} \frac{d M_{0,n} (\rho) \varphi_{\varepsilon}(|\rho|)}{\rho} \underset{n \rightarrow \infty}{\longrightarrow} 
 \int_{[-A',A']} \frac{d M_0 (\rho) \varphi_{\varepsilon}(|\rho|)}{\rho},$$
 where $\varphi_{\varepsilon}$ is a smooth function with values in $[0,1]$, equal to zero in $[0, \varepsilon/2]$ and $[A'-\varepsilon/2, A']$ and equal to $1$ on $[\varepsilon, A' - \varepsilon]$.
 Now, the change we make in the integrals by introducing $\varphi_{\varepsilon}$ is at most 
 $$\int_{[-A',-A'+ \varepsilon] \cup [-\varepsilon,\varepsilon] \cup  [A'-\varepsilon, A']} \frac{d M_{0,n}(\rho)}{ |\rho|},$$
and the similar integral with $M_0$. 
It suffices to check that the upper limit of the integral with $M_{0,n}$ when $n$ goes to infinity, and the integral with $M_0$,  tend to zero with $\varepsilon$.
For the parts of the integrals on $[-A',-A'+ \varepsilon]$ and $[A'-\varepsilon, A']$, this is directly deduced from the vague convergence of $M_{0,n}$ towards $M_0$, and the fact
that $M_0$ has no atom at $-A'$ or $A'$ by assumption. 
For the integral on $[-\varepsilon, \varepsilon]$, we get 
$$\int_{[-\varepsilon,\varepsilon]} \frac{d M_{0,n}(\rho)}{ |\rho|}
\leq \varepsilon \int_{\mathbb{R}} \frac{d M_{0,n}(\rho)}{ \rho^2}
= 2 \varepsilon \int_{\mathbb{R}} \frac{|M_{0,n}(x)|}{ |x|^3} dx,$$
and a similar inequality for $M_0$, which proves that the integrals tend uniformly to zero with $\varepsilon$. 
The last step consists in proving the convergence in law of the finite-dimensional marginals of 
$$\prod_{\rho \in X_n \cap [-A',A']} \left(1- \frac{s}{\rho} \right)$$
toward those of 
$$\prod_{\rho \in X \cap [-A',A']} \left(1- \frac{s}{\rho} \right).$$
We notice that changing $s$ to $s'$ changes the first product by at most
\begin{align*}
& \sum_{\rho \in X_n \cap [-A',A']} \frac{|s - s'|}{|\rho|} \prod_{\rho' \in X_n \cap [-A',A'], \rho' \neq \rho} \left( 1 + \frac{\max (|s|, |s'|)}{|\rho|} \right)
\\ & \leq |s-s'| L_n \mu_n^{-1} \left( 1 + \frac{\max (|s|, |s'|)}{\mu_n} \right)^{L_n - 1}
\end{align*}
where $L_n$ is the cardinality of $X_n \cap [-A',A']$ and $\mu_n$ is the minimum of $|\rho|$ for $\rho \in X_n \cap [-A',A']$.  
Since $X_n$ tends in law to $X$, and $X$, $X_n$ are a.s. locally finite with no point at zero, we deduce that $(L_n)_{n \in \mathbb{N}}$ and $(\mu_n^{-1})_{n \in \mathbb{N}}$ are tight families of random variables, and 
then for $s, s'$ in a given compact set $K$, 
$$\left| \prod_{\rho \in X_n \cap [-A',A']} \left(1- \frac{s}{\rho} \right) - \prod_{\rho \in X_n \cap [-A',A']} \left(1- \frac{s'}{\rho} \right)  \right| \leq T_{K,n} |s - s'| $$
where $(T_{K,n})_{n \in \mathbb{N}}$ is a tight family of random variables. 
Similarly, we have a finite random variable $T_K$ such that 
$$\left| \prod_{\rho \in X \cap [-A',A']} \left(1- \frac{s}{\rho} \right) - \prod_{\rho \in X \cap [-A',A']} \left(1- \frac{s'}{\rho} \right)  \right| \leq T_{K} |s - s'|.$$
Hence, if $(s_1, \dots, s_p)$ is a finite set of complex numbers, 
it is not difficult to deduce the convergence in law 
$$\left(\prod_{\rho \in X_n \cap [-A',A']} \left(1- \frac{s_j}{\rho} \right) \right)_{1 \leq j \leq p} 
\underset{n \rightarrow \infty}{\longrightarrow}  \left(\prod_{\rho \in X \cap [-A',A']} \left(1- \frac{s_j}{\rho} \right) \right)_{1 \leq j \leq p} $$
as soon as we have, for all sufficiently large integers $k \geq 1$, the similar convergence in law when the $s_j$'s are shifted by (say) $i/k$. 
It is then sufficient to check the convergence in law above for $s_1, \dots, s_p \notin \mathbb{R}$. 
In this case, by taking the logarithm and considering the Fourier transform, and by taking into account the fact that $X$ and $X_n$ have a.s. no point at zero (because of our assumptions),  
 we deduce that it is enough to check, for all $\lambda_1, \dots, \lambda_p, \lambda'_1, \dots, \lambda'_p \in \mathbb{R}$, the convergence in law: 
$$\sum_{\rho \in X_n \cap [-A',A']}  g (\rho)  \underset{n \rightarrow \infty}{\longrightarrow}  \sum_{\rho \in X \cap [-A',A']}  g (\rho),$$
where 
$$g : x \mapsto \sum_{j=1}^p \left( \lambda_j \Re \log \left( 1 - \frac{s_j}{x} \right) + \lambda'_j \Im \log \left( 1 - \frac{s_j}{x} \right)  \right) \mathds{1}_{x \neq 0},$$
and where we take the principal branch of the logarithm for the imaginary part, which is possible since $1 - \frac{s_j}{x} \notin \mathbb{R}$. 
Since $g$ is continuous everywhere except at zero, we deduce that for all $\varepsilon > 0$, there exists a continuous function with compact support $g_{\varepsilon}$, 
equal to $g \mathds{1}_{[-A',A']}$ everywhere except at intervals of length $\varepsilon$ around $-A'$, $0$ and $A'$. 
From the convergence of the point process $X_n$ toward $X$, we know that 
$$\sum_{\rho \in X_n}  g_{\varepsilon} (\rho)   \underset{n \rightarrow \infty}{\longrightarrow}  \sum_{\rho \in X}  g_{\varepsilon} (\rho)$$
in distribution. 
If $X_n$ and $X$ have no point in $[x-\varepsilon/2, x + \varepsilon/2]$ for $x \in \{-A',0,A'\}$, the sums do not change when we replace $g_{\varepsilon}$ by $g$. 
Hence, it is enough to check that the lower limit,  when $n$ goes to infinity, of the probability of this event converges to $1$ when $\varepsilon$ goes to zero. 
This is a direct consequence of the convergence of $X_n$ to $X$ and the fact that the point process $X$ has a.s. no point in $\{-A',0,A'\}$ by the assumptions on $X$ and the choice of $A'$. 

We can now deduce, from Proposition \ref{fnA}, the convergence in law of $$s  \mapsto \lim_{A \rightarrow \infty} 
\exp \left( s \int_{[-A,A]} \frac{d M_{0,n}(\rho)}{\rho} \right)\prod_{\rho \in X_n\cap [-A,A]} \left(1- \frac{s}{\rho} \right)$$
towards 
$$s  \mapsto \lim_{A \rightarrow \infty} 
\exp \left( s \int_{[-A,A]} \frac{d M_{0}(\rho)}{\rho} \right)\prod_{\rho \in X\cap [-A,A]} \left(1- \frac{s}{\rho} \right).$$
 In order to complete the proof of Theorem \ref{main}, It only remains to check that the zeros of these holomorphic functions are exactly given by the sets $X_n$ and $X$. 
 This comes from Proposition \ref{weierstrass}, its proof, and the fact that  the zeros of the 
 Weierstrass products 
 $$\prod_{\rho \in X_n} \left(1- \frac{s}{\rho} \right) e^{s/\rho}, \; \prod_{\rho \in X} \left(1- \frac{s}{\rho} \right) e^{s/\rho}$$
 are respectively given by $X_n$ and $X$, when 
 $$\sum_{\rho \in X_n} \rho^{-2} < \infty, \; \sum_{\rho \in X} \rho^{-2} < \infty.$$
 
 In order to deduce Corollary \ref{coro} from Theorem \ref{main}, we use the fact that  for a holomorphic function $h$ which does not vanish at $s$, 
 $$\frac{h'(s)}{h(s)} = \frac{1}{2 \pi h(s)} \int_0^{2 \pi} h(s +  e^{i \theta}) e^{-i \theta} d \theta. $$
 Hence, a uniform convergence of entire functions on compact sets implies a uniform convergence of their logarithmic derivative on compact sets where the limiting function 
 does not vanish. This gives Corollary \ref{coro}, provided that we check that the logarithmic derivative of 
 $$s  \mapsto \lim_{A \rightarrow \infty} 
\exp \left( s \int_{[-A,A]} \frac{d M_{0,n}(\rho)}{\rho} \right)\prod_{\rho \in X_n\cap [-A,A]} \left(1- \frac{s}{\rho} \right),$$
is equal to 
$$s  \mapsto \lim_{A \rightarrow \infty} 
\left( \int_{[-A,A]} \frac{d M_{0,n}(\rho)}{\rho}  + \sum_{\rho \in X_n\cap [-A,A]} \frac{1}{s-\rho} \right),$$
and that 
the logarithmic derivative of 
$$s  \mapsto \lim_{A \rightarrow \infty} 
\exp \left( s \int_{[-A,A]} \frac{d M_{0}(\rho)}{\rho} \right)\prod_{\rho \in X \cap [-A,A]} \left(1- \frac{s}{\rho} \right),$$
is equal to 
$$s  \mapsto \lim_{A \rightarrow \infty} 
 \left(\int_{[-A,A]} \frac{d M_{0}(\rho)}{\rho}  + \sum_{\rho \in X \cap [-A,A]} \frac{1}{s-\rho} \right).$$
These equalities are obvious if we remove the limit in $A$, and then it is enough to check that we can exchange the logarithmic derivative and the limit $A \rightarrow \infty$. 
This is due to the fact that the convergence stated in Proposition \ref{weierstrass} holds uniformly on compact sets.

\section{Proof of Proposition \ref{crit}} \label{criterion}
Since $|M_{0,n} (x)|$ is equal to zero for $x \in [-1,1]$ and is bounded by $C ( 1 + |x|) |x|^{\nu}$ for $|x| > 1$, the boundedness, with respect to $n$, of 
$$\int_{\mathbb{R}} \frac{|M_{0,n}(x)| \varphi(|x|)}{|x|^3} dx $$
is immediate if we take $\varphi(|x|) = \mathcal{O}(1 + |x|^{(1 - \nu)/2})$. 
Similarly, one gets immediately: 
$$\int_{\mathbb{R}} \frac{|M_{0}(x)| \varphi(|x|)}{|x|^3} dx < \infty. $$
We have $M_{0,n} (x) = 0$ for $x \in [-1,1]$ and then $M_{0} (x) = 0$  for $x \in (-1,1)$ since $M_{0,n}$ converges vaguely to $M_0$. 
Hence, 

$$\int_{[-1,1]} \frac{|M_{0,n} (x) - M_{X_n}(x) | \varphi(|x|)}{x^2} dx
\leq \frac{M_{X_n}([-1,1]) \varphi(1)}{\mu_n^2}$$

$$\int_{(-1,1)} \frac{|M_{0} (x) - M_{X}(x) | }{x^2} dx
\leq \frac{M_{X}((-1,1)) }{\mu^2}$$ 
$$\sup_{x \in \mathbb{R}} \frac{|M_{0,n} (x) - M_{X_n}(x) |\varphi(|x|)}{|x|}
\leq  \frac{M_{X_n}([-1,1])  \varphi(1)}{\mu_n},$$
where $\mu_n$ is the smallest modulus of a point of $X_n$ and $\mu$ is the smallest modulus of a point of $X$.

These variables form a tight family of random variables, since $X_n$ converges in law to $X$ and $X_n$, $X$ are locally finite, with a.s. no point at zero. In particular, the tightness of 
$(\mu_n^{-1})_{n \in \mathbb{N}}$, which are finite random variables,  is deduced from the fact that for $A > 0$, 
$$\underset{n \rightarrow \infty}{\lim \sup} \, \mathbb{P} [ \mu_n^{-1} > A] = 
\underset{n \rightarrow \infty}{\lim \sup}\,  \mathbb{P} [ M_{X_n} ((-1/A, 1/A)) > 0] 
$$ $$\leq \mathbb{P} [ M_{X} ((-2/A, 2/A)) > 0]  \underset{ A \rightarrow \infty}{\longrightarrow}   0. $$

In order to prove the proposition, it is now enough to check that

$$\int_{ \mathbb{R} \backslash (-1,1)} \frac{|M_{0} (x) - M_{X}(x) | \varphi(|x|)}{x^2} dx
< \infty,$$ 

$$ \frac{|M_0 (x) - M_X(x) |}{|x|} \underset{|x| \rightarrow \infty}{\longrightarrow} 0,$$
and that the following families are tight for a suitable choice of $\varphi$: 
$$\left(\int_{\mathbb{R} \backslash [-1,1]} \frac{|M_{0,n} (x) - M_{X_n}(x) | \varphi(|x|)}{x^2} dx \right)_{n \in \mathbb{N}},$$

 $$\left(\sup_{x \in  \mathbb{R} \backslash [-1,1]} \frac{|M_{0,n} (x) - M_{X_n}(x) |\varphi(|x|)}{|x|} \right)_{n \in \mathbb{N}}.$$

We have 
\begin{align*}
\mathbb{E} [ |M_{0,n} (x) - & M_{X_n}(x) | ]  \leq | M_{0,n} (x) - \mathbb{E} [ M_{X_n}(x) ] | + \sqrt{ \operatorname{Var} (M_{X_n}(x)) } 
\\ & \leq C (1 + |x|^{\alpha}) + [C (1 + |x|^{2 \alpha})]^{1/2} \leq 2 (1 + C) (1+|x|^{\alpha})
\end{align*}
which proves the tightness of the family of integrals if $\varphi$ goes to zero sufficiently slowly, because $\alpha < 1$. The finiteness of the integral involving $M_0$ and $M_X$ is proven similarly. Let us now prove the tightness of the supremum: by symmetry on the assumptions, it is enough to consider the supremum on $(1, \infty)$. 
For $K > 0$, $p \geq 1$ integer, and $x \in [p^K, (p+1)^K]$, we observe that 
\begin{align*}
&  M_{0,n} (x) - M_{X_n}(x) \leq M_{0,n} ((p+1)^K) - M_{X_n} (p^K) 
\\ & \leq  M_{0,n} (p^K) - M_{X_n} (p^K) + C (1 + (p+1)^K - p^K)  (p^{K \nu} + (p+1)^{K \nu}),
\end{align*}
and 
\begin{align*}
&  M_{0,n} (x) - M_{X_n}(x) \geq M_{0,n} (p^K) - M_{X_n} ((p+1)^K) 
\\ & \geq  M_{0,n} ((p+1)^K) - M_{X_n} ((p+1)^K) - C (1 + (p+1)^K - p^K)  (p^{K \nu} + (p+1)^{K \nu}),
\end{align*}
Hence, 
\begin{align*} 
& \sup_{x \in  [p^K, (p+1)^K]} \frac{|M_{0,n} (x) - M_{X_n}(x) |\varphi(|x|)}{|x|}
\\ & \leq \frac{\varphi((p+1)^K)}{p^K}  \left( | M_{0,n} (p^K) - M_{X_n} (p^K)|  .+ | M_{0,n} ((p+1)^K) - M_{X_n} ((p+1)^K)| \right)
\\ &  + C \, \frac{\varphi((p+1)^K)}{p^K}  (1 + (p+1)^K - p^K)  (p^{K \nu} + (p+1)^{K \nu})
\end{align*} 
and 
\begin{align*} 
& \sup_{x \in (1, \infty)} \frac{|M_{0,n} (x) - M_{X_n}(x) |\varphi(|x|)}{|x|}
\\ & \leq  \sum_{p=1}^{\infty}  \frac{\varphi((p+1)^K)}{p^K}  \left( | M_{0,n} (p^K) - M_{X_n} (p^K)|  .+ | M_{0,n} ((p+1)^K) - M_{X_n} ((p+1)^K)| \right)
\\ &  + C \, \sup_{p \in \mathbb{N}} \frac{\varphi((p+1)^K)}{p^K}  (1 + (p+1)^K - p^K)  (p^{K \nu} + (p+1)^{K \nu})
\end{align*} 
The second term is independent of $n$, and finite as soon as $K  < 1/\nu$ and $\varphi$ goes to infinity sufficiently slowly. 
The expectation of the sum is bounded by
$$\sum_{p=1}^{\infty}  \frac{\varphi((p+1)^K)}{p^K}   \left( 2 (1+C) (2 + p^{K \alpha} + (p+1)^{K \alpha} )\right). $$
This bound is finite for $\varphi$ going to infinity sufficiently slowly, if 
$$K( \alpha - 1) < -1.$$
We then have the tightness we look for, as soon as we can find $K \in (0,1/\nu)$ satisfying this inequality. 
This occurs for 
$$ (1/\nu) (\alpha - 1) < -1,$$
which is true since we assume $\alpha + \nu < 1$. 
Applying the tightness  just proven to the constant sequence of point processes equal to $X$, which satisfies the assumptions made on $(X_n)_{n \in \mathbb{N}}$ for 
$M_{0,n} = M_0$, we deduce that
$$\sup_{x \in  \mathbb{R}} \frac{|M_{0} (x) - M_{X}(x) |\varphi(|x|)}{|x|}  < \infty$$
for some suitable function $\varphi$. This implies: 
$$ \frac{|M_0 (x) - M_X(x) |}{|x|} \underset{|x| \rightarrow \infty}{\longrightarrow} 0.$$

\section{Proof of Theorem \ref{converse}} \label{proof-converse}
Let $(z_j)_{j \in \mathbb{N}}$ be a dense sequence of distinct elements of $D$, and for all $j \in \mathbb{N}$, let $r_j$ be the minimum of $1$ and the distance between $z_j$ and the complement of $D$: we have $r_j > 0$ since $D$ is open. 
Let $\Phi$ be the  function from $\mathcal{C}(D, \mathbb{C})$ to $\mathbb{C}^{\mathbb{N}} \times \mathbb{R}^{\mathbb{N}}$, 
defined by 
$$\Phi(g) := \left( (g(z_j))_{j \in \mathbb{N}}, \left( \sup_{ m \in \mathbb{N},  |z_m - z_j| \leq r_j/2}  |g(z_m)| \right)_{j \in \mathbb{N}} \right).$$
If $\mathcal{C}(D, \mathbb{C})$ is endowed with the topology of uniform convergence on compact sets, and $\mathbb{C}^{\mathbb{N}} \times \mathbb{R}^{\mathbb{N}}$ is endowed with 
the topology of pointwise convergence, then the function $\Phi$ is continuous. 
Hence, $(\Phi(f_n))_{n \in \mathbb{N}}$ converges to $(\Phi(f))_{n \in \mathbb{N}}$ in distribution. Moreover, the space 
$\mathbb{C}^{\mathbb{N}} \times \mathbb{R}^{\mathbb{N}}$ is separable, since the set of rational couples of sequences with finitely many non-zero components is dense 
for the topology of pointwise convergence. By Skorokhod's representation theorem (see \cite{B99}, Theorem 6.7), there exist random variables 
$h$ and  $(h_n)_{n \in \mathbb{N}}$, with values in  $\mathbb{C}^{\mathbb{N}} \times \mathbb{R}^{\mathbb{N}}$, 
such that $h_n$ has the same law as $\Phi(f_n)$, $h$ has the same law as $\Phi(f)$, and 
$(h_n)_{n \in \mathbb{N}}$ almost surely converges to $h$, for the topology of pointwise convergence. 
Since $f_n$ and $f$ are almost surely continuous, they are almost surely uniformly continuous on the closed disc of center $z_j$ and radius $r_j/2$, for all $j \in \mathbb{N}$. 
Hence, $\Phi(f_n)$ and $\Phi(f)$ are almost surely in the measurable set: 
$$\{ ((u_j)_{j \in \mathbb{N}}, (v_j)_{j \in \mathbb{N}}), \; \forall m \in \mathbb{N}, 
\forall p \in \mathbb{N}, \exists q \in \mathbb{N}, 
\forall (k, \ell) \in  E_{m,q}, |u_{k} - u_{\ell} | \leq 2^{-p} \}$$
where 
$$E_{m,q} := \{ (k, \ell) \in \mathbb{N}^2, |z_{k} - z_m| \leq r_m/2,   |z_{\ell} - z_m| \leq r_m/2, |z_{\ell} - z_{k} | \leq 2^{-q}\}.$$
 Since $h_n$ has the same law as $\Phi(f_n)$ and $h$ has the same law as $\Phi(f)$, 
 $h_n$ and $h$ are almost surely in the measurable set above, which means that the map 
 from $\{(z_j)_{j \in \mathbb{N}}\}$ to $\mathbb{C}$ sending $z_j$ to the $j$-th coordinate of $h_n$ or $h$
 is uniformly continuous on the intersection of $\{(z_j)_{j \in \mathbb{N}}\}$  and the  closed disc of center $z_m$ and radius $r_m/2$, for all $m \in \mathbb{N}$. 
 Hence, for each $n \in \mathbb{N}$, there exists almost surely a unique continuous function $\tilde{f}_n$ on $D$ such that 
 $\tilde{f}_n (z_j)$ is equal to the $j$-th coordinate of $h_n$ for all $j \in \mathbb{N}$, and there exists a similar function $\tilde{f}$ corresponding to $h$. 
 Moreover, by writing the definitions of the suitable limits along a sequence of $j$'s such that $z_j$ tends to $z$, one checks that for all $z \in D$, 
 there exists a measurable function $L_z$ from $\mathbb{C}^{\mathbb{N}} \times \mathbb{R}^{\mathbb{N}}$ to $\mathbb{C}$, equal to the $j$-th coordinate when $z = z_j$, and  such that 
 $$ f_n (z) = L_z (\Phi(f_n)), \; f (z) = L_z (\Phi(f)), \;  \tilde{f}_n (z) = L_z( h_n), \; \tilde{f}(z) = L_z (h).$$
  For $n \in \mathbb{N}$,  $z'_1, \dots, z'_p \in D$, $B_1, \dots, B_p$ Borel sets of $\mathbb{C}$, the event
  $$\{ \forall k \in \{1, \dots, p\},  \tilde{f}_n(z'_k) \in B_k\} = \bigcap_{k=1}^p \{ h_n \in L_{z'_k}^{-1} (B_j) \}$$
  is in the $\sigma$-algebra of the probability space where the random variables $h_n$ and $h$ are defined, and then $\tilde{f}_n$ 
  is a $\mathcal{C}(D, \mathbb{C})$-valued random variable defined on this probability space. Similarly, $\tilde{f}$ is a random variable defined on the same space. 
  Moreover, since $L_{z_j}$ is given by the $j$-th coordinate, and $h_n$ has the same law as $\Phi(f_n)$, 
\begin{align*} 
\mathbb{P} [ \forall j \in \{1, \dots, p\},  \tilde{f}_n(z_j) \in B_j ]  & = \mathbb{P} [ \forall j \in \{1, \dots, p\}, (h_n)_{j} \in B_j ] 
\\ & = \mathbb{P} [ \forall j \in \{1, \dots, p\}, (\Phi(f_n))_{j} \in B_j ]  \\ & = \mathbb{P} [ \forall j \in \{1, \dots, p\}, f_n(z_j) \in B_j ] 
\end{align*}
which shows that $(\tilde{f}_n (z_j))_{1 \leq j \leq p}$ has the same distribution as $(f_n (z_j))_{1 \leq j \leq p}$, and then 
$\tilde{f}_n$ has the same distribution as $f_n$, since by density of $(z_j)_{j \in \mathbb{N}}$ in $D$, the $\sigma$-algebra of $\mathcal{C}(D, \mathbb{C})$ corresponding to the topology 
of uniform convergence on compact subsets of $D$ is generated by the values of the functions at $z_j$ for $j \in \mathbb{N}$. 
Similarly, $\tilde{f}$ has the same law as $f$. In particular, $\tilde{f_n}$ and $\tilde{f}$ are almost surely holomorphic on $D$, since the 
holomorphic functions form a measurable subset of $\mathcal{C}(D, \mathbb{C})$ (a continuous function on $D$ is holomorphic if and only 
if some countable family of contour integrals of this function are equal to zero). 

 On the other hand, by definition of the map $\Phi$, the image of any continuous function by $\Phi$ is in the set 
  $$ H := \left\{ ((u_j)_{j \in \mathbb{N}}, (v_j)_{j \in \mathbb{N}} ), \forall j \in \mathbb{N}, v_j = \sup_{ m \in \mathbb{N},  |z_m - z_j| \leq r_j/2}  |u_m| \right\}.$$
  Since $h_n$ and $h$ are distributed as $\Phi(f_n)$ and $\Phi(f)$, we deduce that $h_n, h \in H$ almost surely. 
  Moreover,  for all $j \in \mathbb{N}$, $\tilde{f}_n (z_j) = L_{z_j} (h_n)$ is the $j$-th coordinate of $h_n$,  and then from $h_n \in H$, we deduce that  
  $h_n = \Phi (\tilde{f}_n )$. Similarly, $h = \Phi(\tilde{f})$. At this stage, we have random continuous functions $\tilde{f}_n$, $\tilde{f}$, almost surely holomorphic, with the same law as $f_n$ and $f$, such that $ h_n = \Phi(\tilde{f}_n)$ converges to $h = \Phi(\tilde{f})$ almost surely, which means that for all $j \in \mathbb{N}$, 
  $$\tilde{f}_n(z_j)  \underset{n \rightarrow \infty}{\longrightarrow} \tilde{f} (z_j)$$
 $$ \sup_{ m \in \mathbb{N},  |z_m - z_j| \leq r_j/2}  |\tilde{f}_n(z_m)|   \underset{n \rightarrow \infty}{\longrightarrow}
  \sup_{ m \in \mathbb{N},  |z_m - z_j| \leq r_j/2}  |\tilde{f}(z_m)|.$$
  By continuity, the last convergence implies that for all $j \in \mathbb{N}$, 
  all the functions $\tilde{f}_n$, $\tilde{f}$ are uniformly bounded on the closed disc of center $z_j$ and radius $r_j/2$, independently of $n$. 
  By Cauchy's formula, $\tilde{f}'_n$, $\tilde{f}'$ are uniformly bounded on the closed disc of center $z_j$ and radius $r_j/3$, which 
  shows that $(\tilde{f}_n)_{n \in \mathbb{N}}$, $\tilde{f}$ are equicontinuous on this disc. We deduce a uniform convergence of 
  $\tilde{f}_n$ to $\tilde{f}$ on this disc, from the convergence of $\tilde{f}_n(z_m)  $ towards $\tilde{f}(z_m)$ for all $m$ and the density of $(z_m)_{m \in \mathbb{N}}$. 
Hence, $\tilde{f}_n$ converges to $\tilde{f}$ uniformly on any compact set $K$ of $D$, since $D$ is  covered by the open discs of center $z_j$ and radius $r_j/3$, 
from which one can extract a finite covering of $K$. 
Therefore, in the statement of Theorem \ref{converse}, we can assume that $f_n$ almost surely converges to $f$, uniformly on compact sets of $D$.

Let $K$ be a compact subset of $D$. Since the zeros of $f$ have no accumulation point, there are finitely many of them in $K$, denoted $z'_1, \dots, z'_p$, with respective multiplicity
$m_1, \dots, m_p$.  For $\varepsilon > 0$ small enough, the closed balls  $\mathcal{B}_1, \dots, \mathcal{B}_p$ centered in $z'_1, \dots, z'_p$, with radius $\varepsilon$, are included in $D$, do not overlap, 
and each of them has only one zero of $f$.
By Cauchy's formula, $f'_n$ converges uniformly to $f'$ on $K$. Moreover, since the only zero of $f$ on $\mathcal{B}_j$ is $z'_j$, we have by compactness that $|f|$ is bounded from below
by a non-zero quantity on the circle  $\mathcal{C}_j$ given by the boundary of  $\mathcal{B}_j$. Hence, $f'_n/f_n$ converges uniformly to $f'/f$ on $\mathcal{C}_j$, which implies, by 
integrating on this circle, that for $n$ large enough, $f_n$ has exactly $m_j$ zeros in $\mathcal{B}_j$, counted with multiplicity. 
Let $K'$ be the compact set obtained from $K$ by removing the open balls $\overset{\circ}{\mathcal{B}}_1, \dots, \overset{\circ}{\mathcal{B}}_p$. 
The function $f$ has no zero on $K'$, and by compactness, $|f|$ is bounded from below by a non-zero quantity on $K'$, which implies that $f_n$ has also no zero on $K'$ for $n$ large enough. 
We have then proven that for $n$ large enough, the multiset of zeros of $f_n$ in $K \cup \mathcal{B}_1 \cup \dots \cup \mathcal{B}_p$, counted with multiplicity, can be obtained from the multiset of  zeros of $f$ in $K$ by moving each point by at most $\varepsilon$. 
Let $g$ be a nonnegative, continuous function on $D$, supported in $K$. 
If $\varepsilon$ is smaller than half the distance $\delta$ from $K$ to the complement of $D$, which is strictly positive, then 
$$K' := \left\{z \in \mathbb{C}, \inf_{z' \in K} |z - z'| \leq \delta/2\right\}$$
is compact, included in $D$, and contains  $K \cup \mathcal{B}_1 \cup \dots \cup \mathcal{B}_p$. 
The function $g$ is uniformly continuous on $K'$, which implies that 
$$\sigma(\varepsilon) := \sup_{z, z' \in K', |z-z'| \leq \varepsilon}  |g(z)- g(z')|$$
tends to zero with $\varepsilon$.  
We have 
$$\left| \sum_{\rho \in K \cup \mathcal{B}_1 \cup \dots \cup \mathcal{B}_p, f_n(\rho) = 0} g(\rho) \mu_{f_n}(\rho) - 
\sum_{\rho \in K, f(\rho) = 0} g(\rho) \mu_{f}(\rho) \right| \leq \sigma( \varepsilon) \sum_{j=1}^p m_j$$
for $n$ large enough. 
Since $g$ is supported in $K$, we deduce 
$$\left| \sum_{ \rho \in D, f_n(\rho) = 0} g(\rho) \mu_{f_n}(\rho) - 
\sum_{\rho \in D, f(\rho) = 0} g(\rho) \mu_{f}(\rho) \right| \leq \sigma( \varepsilon)  \sum_{j=1}^p m_j$$
for $n$ large enough. 
This completes the proof of Theorem \ref{converse}, by letting $\varepsilon \rightarrow 0$.

\section{Examples} \label{examples}

\subsection{Rotationally invariant circular ensembles}

For $n \in \mathbb{N}$, we consider  a random set  $(\lambda^{(n)}_j)_{1 \leq j \leq n}$ of $n$ points on the unit circle, and we define the following polynomial, 
corresponding to the characteristic polynomial of a random matrix whose spectrum is given by $(\lambda^{(n)}_j)_{1 \leq j \leq n}$: 
$$Z_n (x) := \prod_{j=1}^n ( 1 - x/ \lambda^{(n)}_j).$$
We assume that almost surely, $\lambda^{(n)}_j \neq 1$ for all $j$, and we then define, for all $z \in \mathbb{C}$:  
$$\xi_n (s) := Z_n(e^{2 i \pi s /n})/ Z_n(1).$$

It is natural to look for conditions under which the random holomorphic function $\xi_n$ converges in law to a limit when $n$ goes to infinity. 
We have the following: 
\begin{proposition} \label{invariantcircle}
We assume that the set $ \Lambda_n := \{ (\lambda^{(n)}_j)_{1 \leq j \leq n} \}$ is rotationally invariant in distribution, and that the point process of renormalized eigenangles: 
$$ X_n := \{x \in \mathbb{R}, e^{2 i \pi x /n} \in \Lambda_n\}$$
converges in law to a limiting simple point process $X$. Moreover, we assume that the expectation of the number of points of $X$ in a finite interval 
is equal to the length of this interval (or proportional to it), and that  for some $C > 0$, $\alpha \in (0,1)$,  the variance of the number of points of $\Lambda_n$ lying in an arc of length $\theta$ is smaller than $C (1 + (n \theta)^{2 \alpha})$, uniformly in $n$ and $\theta$. Then, $1 \notin \Lambda_n$ almost surely, and
we have the following convergence in law of random holomorphic functions, for the topology of uniform convergence on compact sets: 
$$(\xi_n (s))_{s \in \mathbb{C}} \underset{n \rightarrow \infty}{\longrightarrow } (\xi_{\infty} (s))_{s \in \mathbb{C}} $$
where for $s \in \mathbb{C}$, $\xi_{\infty}(s)$ is  well-defined by the limit: 
$$\xi_{\infty}(s) := e^{i \pi s} \underset{A \rightarrow \infty}{\lim} \prod_{\rho \in X \cap [-A,A]} \left(1 - \frac{s}{\rho} \right).$$
\end{proposition}

\begin{proof}
Since $\Lambda_n$ is rotationally invariant in law and has a given finite number of  points, $1 \notin \Lambda_n$ almost surely and 
$X_n$ is translation invariant in law, with $0 \notin X_n$. This invariance and the fact that $X_n$ has $n$ points in an interval of size $n$ implies that 
$\mathbb{E} [ M_{X_n} (I)] $ is the length of $I$ for any finite interval $I$. By assumption, 
$\mathbb{E} [ M_{X} (I)] $ is also the length of $I$, and then $0 \notin X$ almost surely.

An elementary computation, detailed in 
\cite{CNN17}, gives 
$$\xi_n (s) = e^{i \pi s} \lim_{A \rightarrow \infty} \prod_{\rho \in X_n \cap [-A,A]} \left(1 - \frac{s}{\rho} \right).$$
Hence, it is enough to find measures $M_0$ and $(M_{0,n})_{n \in \mathbb{N}}$ such that Proposition \ref{crit} is satisfied. 
We take $M_{0,n}$ and $M_0$ equal to the Lebesgue measure, restricted to $\mathbb{R} \backslash [-1,1]$. 
The vague convergence and the fact that the measures do not charge $[-1,1]$ hold obviously. The bound on the measure of $[x,y]$ is satisfied for $\nu > 0$ arbitrarily small. 
Moreover, we have 
$$|M_{0,n} (x) - \mathbb{E} [ M_{X_n} (x)] |, \;  |M_{0} (x) - \mathbb{E} [ M_{X} (x)] | \leq 1$$
and by assumption
$$\operatorname{Var} (M_{X_n}(x)) \leq C (1 + (2 \pi |x|)^{2 \alpha})$$
for some $\alpha \in (0,1)$. Taking $\nu$ sufficiently small, we are done, provided that we check that 
$$\operatorname{Var} (M_{X}(x)) = \mathcal{O} ( 1 + |x|^{2 \alpha'})$$
for some $\alpha' \in (0,1)$. 
Now, we have, for $x \geq 0$, 
$$M_X(x) \leq \int_{\mathbb{R}} g_x dM_X$$
where $g_x$ is a continuous function from $\mathbb{R}$ to $[0,1]$, with support included in $[-1, x+1]$, and equal to $1$ on $[0,x]$. 
Hence,
\begin{align*} 
\mathbb{E} [ (M_{X}(x))^2 ] 
& = \underset{ A \rightarrow \infty}{\lim} \mathbb{E} [ \min((M_X(x))^2, A) ] 
\\ &  \leq  \underset{ A \rightarrow \infty}{\lim}  \, \mathbb{E} \left[ \min \left( \left( \int_{\mathbb{R}} g_x dM_X \right)^2, A \right) \right]
\\ & = \underset{ A \rightarrow \infty}{\lim} \,  \underset{n \rightarrow \infty}{\lim} \, \mathbb{E} \left[ \min \left( \left( \int_{\mathbb{R}} g_x dM_{X_n} \right)^2, A \right) \right]
\\ & \leq \underset{n \rightarrow \infty}{\lim \inf} \, \mathbb{E} [ (M_{X_n} ([-1,x+1]))^2]. 
\end{align*}
Now, by triangle inequality, 
$$\left(\mathbb{E} [ (M_{X_n} ([-1,x+1]))^2] \right)^{1/2}
\leq \left(\mathbb{E} [ (M_{X_n} (-1))^2] \right)^{1/2}
+ \left( \mathbb{E} [ ( M_{X_n} (x+1))^2] \right)^{1/2}$$
with $$\mathbb{E} ( M_{X_n} (x+1))^2] = \left(\mathbb{E} [  M_{X_n} (x+1)] \right)^2 + \operatorname{Var} ( M_{X_n} (x+1)) 
=(x+1)^2 + \mathcal{O} ( 1 + x^{2 \alpha}),$$
 $$\mathbb{E} ( M_{X_n} (-1))^2] = \mathcal{O}(1),$$
 uniformly in $n$. 
 Hence
 $$\left(\mathbb{E} [ (M_{X_n} ([-1,x+1]))^2] \right)^{1/2} 
\leq x + \mathcal{O}( 1 + x^{2 \alpha - 1}),$$
and 
$$\mathbb{E} [ (M_{X_n} ([-1,x+1]))^2]
\leq x^2 + \mathcal{O}( 1 + x + x^{2 \alpha} + x^{4 \alpha - 2} )$$
 uniformly in $n$.
 We deduce 
 $$\mathbb{E} [  (M_{X}(x))^2 ] \leq x^2 + \mathcal{O}( 1 + x + x^{2 \alpha} + x^{4 \alpha - 2} ),$$
 and since $\mathbb{E} [M_{X}(x)] = x$ by assumption, 
 $$\operatorname{Var} (M_{X}(x)) =  \mathcal{O}( 1 + x + x^{2 \alpha} + x^{4 \alpha - 2} )$$
 for $x \geq 0$. By symmetry of the setting, we have for all $x \in \mathbb{R}$, 
 $$\operatorname{Var} (M_{X}(x)) =  \mathcal{O}( 1 + |x| + |x|^{2 \alpha} + |x|^{4 \alpha - 2} ),$$
 which gives the desired result by taking 
 $$\alpha' = \max(1/2,\alpha,  2 \alpha - 1)  \in (0,1).$$
\end{proof} 
A particular case where the proposition applies is the following: 
\begin{proposition}
If $\Lambda_n$ is given by $n$ i.i.d. uniform points on the unit circle, and if $X$ is a Poisson point process of intensity $1$, then the convergence given in 
Proposition \ref{invariantcircle} occurs. 
\end{proposition}

\begin{proof}
The condition on the  expectation of the number of points of $X$ is obviously satisfied. The number of points of $\Lambda_n$ in an arc of length $\theta$ is a binomial 
distribution of parameters $n$ and $\theta/2 \pi$ , which gives the condition on the variance satisfied for $\alpha = 1/2$. It remains to check the convergence of the point 
processes. This can be done by using Proposition B.1 of \cite{MNN19}, applied to the restriction $E_n$ of the process $X_n$ to $(-n/2, n/2]$: this restriction has no 
influence on the convergence of the point processes, since the corresponding  test functions have compact support. 
For all $1 \leq r \leq n$,  the $r$-point correlation  function of $E_n$ is well-defined, and constant on $(-n/2, n/2]^r$: since the number of sequences of $r$ distinct points in $E_n$ is $n!/(n-r)!$, the correlation function is $n^{-r} n!/(n-r)!$ if all the points are in $(-n/2,n/2]$ and zero otherwise. 
Since the correlation function is smaller than $1$ and tends to $1$ pointwise when $r$ is fixed and $n \rightarrow \infty$, we have than $E_n$, and then $X_n$, converges in law to a Poisson point process of intensity $1$ when $n \rightarrow \infty$. 
\end{proof}
For any $\beta > 0$, $n \in \mathbb{N}$,  the Circular $\beta$ Ensemble (C$\beta$E) is given by a random set of $n$ points on the unit circle, 
whose joint probability density function is proportional to the power $\beta$ of the product of all the mutual distances between the points: it is clearly rotationally invariant. 
For $\beta = 2$, the C$\beta$E corresponds to the Circular Unitary Ensemble (CUE), i.e. it has the same law as the spectrum of a Haar-distributed unitary matrix of order $n$. 
In \cite{KS09}, Killip and Stoiciu show that for $\Lambda_n$ given by the C$\beta$E, $X_n$ goes in distribution to a limiting point process $X$, called 
the $\operatorname{Sine}_{\beta}$ process: for $\beta = 2$, $X$ coincides with the determinantal sine-kernel point process. 
We then have the following: 
\begin{proposition} \label{CbetaE}
If $\Lambda_n$ is given by the C$\beta$E, and if $X$ is a $\operatorname{Sine}_{\beta}$ process, then the convergence given in 
Proposition \ref{invariantcircle} occurs. 
\end{proposition}
\begin{proof}
The convergence of $X_n$ towards $X$ is proven in \cite{KS09}, and a logarithmic bound of the variance of the number of points of $\Lambda_n$
in an arc is provided in  \cite{NV21}: this shows that we can take $\alpha > 0$ arbitrarily small in Proposition \ref{invariantcircle}. 
In \cite{NV21}, we prove, after suitable rescaling of the point processes by a factor of $2 \pi$, that for $x \geq 0$, 
$$\mathbb{E} [ (M_{X}(x) - x)^2 ] = \mathcal{O}( \log (2 + x)),$$
which implies 
$$\mathbb{E} [ M_X(x) ] =  x + \mathcal{O} \left(\mathbb{E} [ |M_X(x) - x| ]  \right) = x + \mathcal{O}(\sqrt{ \log (2+x)}).$$
The  $\operatorname{Sine}_{\beta}$ process is translation-invariant in distribution, since $\Lambda_n$ is rotationally invariant, and then 
the expectation of the number of points in an interval is proportional to the length of this interval: the coefficient of proportionality is necessarily $1$ by the estimate above:
we have all the ingredients needed to apply Proposition \ref{invariantcircle}. 
\end{proof}
The convergence of holomorphic functions provided by Proposition \ref{CbetaE} has been proven by other methods in \cite{CNN17} for $\beta = 2$, and in 
\cite{VV20} for general $\beta > 0$.

\subsection{Circular ensembles with well-defined correlation functions}

It is still possible go get convergence of the renormalized characteristic polynomial in the case where the point process on the circle is not 
rotationally invariant. If the correlation functions of a point process on the unit circle are normalized with respect to the uniform probability distribution $d \theta/2 \pi$, the following holds: 
\begin{proposition} \label{correlations}
Let us assume, with the notation above,  that the point process $\Lambda_n$ on the unit circle satisfies the following conditions:
\begin{itemize}
\item The exist $C > 0$, $\gamma \in (0,1)$, such that for all $n \geq 1$, the one-point correlation function $\rho^{(n)}_1$ is defined everywhere on the unit circle, and one has for all $\theta \in (-\pi, \pi]$, 
 $$\left| \int_0^{\theta} \left( \rho^{(n)}_1 (e^{i \tau}) - n \right) \frac{d \tau}{2 \pi} \right| \leq C  (1 + (n | \theta|)^{\gamma} ).$$
 \item There exists a finite family of arcs $I_1, \dots, I_k$ covering the unit circle, such that for all $p \in \{1, 2, \dots, k\}$, $n \geq 1$, the two-point correlation function 
 $\rho_2^{(n)}$ is well-defined on $I_p^2$, and for $z_1, z_2 \in I_p$, 
 $$\rho_2^{(n)} (z_1, z_2) \leq \rho_1^{(n)} (z_1)  \rho_1^{(n)} (z_2).$$
\item There exists $\varepsilon \in (0, \pi)$ such that for all  $n \geq 1$, $r \in \{1, 2, \dots, n\}$, the $r$-point correlation function $\rho^{(n)}_r$ of $\Lambda_n$ is well-defined on $J_{\varepsilon}^r$, for  
 $J_{\varepsilon}  := e^{i [-\varepsilon, \varepsilon]}$.  
\item  There exists a function $h$, from $\mathbb{R}$ to $\mathbb{R}_+$, locally integrable, such that for all $r \in \{1, 2, \dots, n\}$, and $x_1, x_2, \dots x_r \in 
K_{n, \varepsilon} := [-n \varepsilon/2 \pi, n \varepsilon/2 \pi]$, 
$$\rho^{(n)}_r(e^{2 i \pi x_1/n}, \dots,  e^{2 i \pi x_r/n}) \leq n^r h(x_1) h(x_2) \dots h(x_r).$$
\item For all $r \geq 1$, there exists a function $\rho_r$
from $\mathbb{R}^r$ to $\mathbb{R}_+$ such that for all $x_1, \dots, x_r \in \mathbb{R}$, 
$$n^{-r} \rho^{(n)}_r(e^{2 i \pi x_1/n}, \dots,  e^{2 i \pi x_r/n}) \underset{n \rightarrow \infty}{\longrightarrow} \rho_r(x_1, \dots, x_r).$$
Then, there exists a point process $X$ on the real line whose $r$-point correlation function $\rho_r$ is well-defined for all $r \geq 1$, 
and the convergence of random holomorphic functions described in 
Proposition \ref{invariantcircle} occurs. 
\end{itemize}

\end{proposition}

\begin{proof}
We first apply Proposition B.1 of \cite{MNN19} to the image $E_n$ of the  restriction of $X_n$ to the interval $[-n \varepsilon/2 \pi, n\varepsilon/2\pi]$, 
by the map $\psi$ from $\mathbb{R}$ to $\mathbb{R}$ given by
$$\psi(x) = \int_0^{x} (1 + h(y)) dy,$$
which is well-defined since $h$ is locally integrable. Notice that $\psi$ is a bijection from $\mathbb{R}$ to $\mathbb{R}$. 
The restriction to $[-n \varepsilon/2 \pi, n\varepsilon/2\pi]$ and the map $\psi$ ensure that the correlation functions of $E_n$ are well-defined and smaller than $1$.
Indeed, if $\tilde{\rho}^{(n)}_r$ denotes the $r$-point correlation function of $E_n$, we have for $t_1, \dots, t_r \in \mathbb{R}$, 
\begin{align*}
& \tilde{\rho}^{(n)}_r (t_1, t_2, \dots, t_r) 
\\ & =  n^{-r} \rho^{(n)}_r (e^{2 i \pi \psi^{\langle - 1 \rangle}(t_1)/n},  \dots, e^{2 i \pi \psi^{\langle - 1 \rangle}(t_r)/n}) \prod_{j=1}^r (1 + h(\psi^{\langle - 1 \rangle}(t_j) ))^{-1}
\\ & \leq  \prod_{j=1}^r  h(\psi^{\langle - 1 \rangle}(t_j) )(1 + h(\psi^{\langle - 1 \rangle}(t_j)))^{-1} \leq 1
\end{align*}  
if $t_1, \dots, t_r \in [\psi(-n \varepsilon/2 \pi), \psi(n \varepsilon/2 \pi)]$
and 
$$\tilde{\rho}^{(n)}_r (t_1, t_2, \dots, t_r)  = 0$$
otherwise. 
The $r$-point correlation functions of $E_n$ converge pointwise 
to $\tilde{\rho}_r$, where 
$$\tilde{\rho}_r(t_1, \dots,  t_r) 
=  \rho_r ( \psi^{\langle - 1 \rangle}(t_1), \dots,  \psi^{\langle - 1 \rangle}(t_r)) \prod_{j=1}^r (1 + h(\psi^{\langle - 1 \rangle}(t_j)))^{-1}.$$
Hence, there exists a point process $E_{\infty}$ such that $E_n$ converges in law to $E_{\infty}$ when $n \rightarrow \infty$, and 
$E_{\infty}$ has $r$-point correlation functions given by $\tilde{\rho}_r$. 
Applying the map $\psi^{\langle -1 \rangle} $, we deduce that the restriction of $X_n$ to $[-n \varepsilon/2 \pi, n\varepsilon/2\pi]$, and then $X_n$ itself, converges 
in law to a point process $X$ with correlation functions $\rho_r$. 

We now apply Proposition \ref{crit} with $M_{0,n}$ and $M_0$ equal to the Lebesgue measure restricted to $\mathbb{R} \backslash [-1,1]$. 
The two first items of Proposition \ref{crit} are immediately satisfied, and the third item is satisfied for $\nu$ arbitrarily small. 
We have, for $n \geq 1$, $x \in (-n/2, n/2]$ 
\begin{align*}
|M_{0,n} (x) - \mathbb{E} [ M_{X_n} (x) ] |
& =  \left| x + \mathcal{O}(1) - \int_0^x n^{-1} \rho^{(n)}_1 (e^{2 i \pi y/ n} )   dy \right| 
\\ & = \mathcal{O}(1) + \left|  \int_0^{2 \pi x/n} n^{-1}  ( \rho^{(n)}_1 (e^{i \tau}) - n)  \,  \left( \frac{n}{2 \pi} \right) d \tau  \right| \\
\\ & = \mathcal{O} \left( 1 + |x|^{ \gamma} \right). 
\end{align*}
The set $X_n$ is $n$-periodic, with exactly $n$ points in each interval of length $n$, and then the estimate above can easily be extended from the interval $(-n/2, n/2]$ to the 
whole real line.

Moreover, for fixed $x \in \mathbb{R}$, we have for $n$ large enough, $|2 \pi x/n| < \pi$ and then 
$$\int_0^{x}  ( n^{-1} \rho^{(n)}_1 (e^{2 i \pi y/n}) - 1) dy 
= \int_0^{2 \pi x/n}  (\rho^{(n)}_1 (e^{\tau}) - 1) \frac{ d \tau}{ 2 \pi} 
= \mathcal{O} ( 1 + |x|^{\gamma}).$$
Moreover, for $n$ large enough, we have $|x| \in [-n \varepsilon/2 \pi, n \varepsilon/2 \pi]$ and then, 
$$n^{-1} \rho^{(n)}_1 (e^{2 i \pi y/n}) \underset{n \rightarrow \infty}{\longrightarrow} \rho_1(y)$$
and
$$n^{-1} \rho^{(n)}_1 (e^{2 i \pi y/n}) \leq h(y).$$
Since $h$ is locally integrable, we can apply dominated convergence and get 
$$ \int_0^x (\rho_1 (y) - 1) dy = \underset{n \rightarrow \infty}{\lim} \, \int_0^{x}  ( n^{-1} \rho^{(n)}_1 (e^{2 i \pi y/n}) - 1) dy 
= \mathcal{O} ( 1 + |x|^{\gamma}).$$
This estimates implies that
$$|M_{0} (x) - \mathbb{E} [ M_{X} (x) ] | = \mathcal{O} \left( 1 + |x|^{ \gamma} \right),$$
which gives the fourth item of Proposition \ref{crit} if $\alpha \in ( \gamma,1) $ and $\nu \in (0, 1 -\alpha)$. 
 
It remains to bound the variance of $M_{X} (x) $ and $M_{X_n}(x)$. 
If for some $p \in \{1, \dots, k\}$, $I$ is a measurable subset of $I_p$ and if $Y$ is the number of points of $\Lambda_n$ in $I$, we have 
$$\mathbb{E} [Y^2] = \int_{I} \rho^{(n)}_1 (z) \mu(dz)  + \int_{I^2} \rho^{(n)}_2 (z_1, z_2) \mu(dz_1) \mu(dz_2)$$
where $\mu$ denotes the uniform probability measure on the unit circle. 
Hence, 
$$\operatorname{Var} (Y) =  \int_{I} \rho^{(n)}_1 (z) \mu(dz)  + 
\int_{I^2} ( \rho^{(n)}_2 (z_1, z_2) -  \rho^{(n)}_1 (z_1)  \rho^{(n)}_1 (z_2)  ) \mu(dz_1) \mu(dz_2),$$
and by the assumption made on $\rho^{(n)}_2$, the second term is nonpositive, which implies 
$$\operatorname{Var} (Y) \leq \mathbb{E} [Y].$$
If $I$ is any measurable subset of the unit circle, and if $Y$ is the number of points of $\lambda_n$ in $I$, we have 
$$\operatorname{Var} (Y)  = \mathbb{E}  \left[ \left( \sum_{p=1}^k ( Y_p - \mathbb{E} [Y_p] ) \right)^{2} \right]$$
where $Y_p$ is the number of points in $I \cap I_p$, and then by Cauchy-Schwarz inequality, 
$$\operatorname{Var} (Y)  \leq  \mathbb{E}  \left[ k \ \sum_{p=1}^k ( Y_p - \mathbb{E} [Y_p] ) ^2 \right]
= k \sum_{p=1}^k \operatorname{Var} (Y_k) \leq k \sum_{p=1}^k \mathbb{E} [Y_k],$$
i.e. 
$$\operatorname{Var} (Y)  \leq k \mathbb{E} [ Y].$$
We then get for $x \in (-n/2, n/2]$, 
$$\operatorname{Var} (M_{X_n}(x)) \leq k \mathbb{E} [ |M_{X_n}(x)|] = k (|x| + \mathcal{O} (  1 + |x|^{ \gamma} )) = \mathcal{O}(1 + |x|)$$
by the bound previously computed on $M_{0} (x) - \mathbb{E} [ M_{X} (x) ] $. 
Since $X_n$ is periodic, the estimate remains true for all $x \in \mathbb{R}$. 
The one and two-point correlation functions of $X$ are pointwise limits of the  correlations functions of $X_n$, which ensures that for $x_1, x_2 \in \mathbb{R}$ of the same sign, 
$$\rho_2 (x_1,x_2) \leq \rho_1 (x_1) \rho_1(x_2).$$
We need $x_1, x_2$ of the same sign in case where $1$ is a boundary of some of the arcs $I_1, \dots, I_k$. 
We deduce, for all $x \in \mathbb{R}$, 
$$\operatorname{Var} (M_{X}(x)) \leq 2 \mathbb{E} [ |M_{X}(x)| ] = \mathcal{O}(1 + |x|).$$
This gives the last item of Proposition \ref{crit} if one takes $\alpha \geq 1/2$. 
\end{proof}

Here is a version of the proposition adapted to the case where $\Lambda_n$ is symmetric with respect to the real axis: 
\begin{proposition} \label{correlations2}
Proposition \ref{correlations} also applies if we modify the assumptions as follows: 
\begin{itemize} 
\item We add the assumption that the points of $\Lambda_n$ are almost surely symmetrically distributed with respect to the real axis. 
\item We take $J_{\varepsilon} = e^{i [0, \varepsilon]}$ instead of $J_{\varepsilon} = e^{i [-\varepsilon, \varepsilon]}$, and 
$K_{n, \varepsilon} = [0, n \varepsilon/2 \pi]$ instead of $K_{n, \varepsilon} = [-n \varepsilon/2 \pi, n \varepsilon/2 \pi]$. 
\item The convergence 
$$n^{-r} \rho^{(n)}_r(e^{2 i \pi x_1/n}, \dots,  e^{2 i \pi x_r/n}) \underset{n \rightarrow \infty}{\longrightarrow} \rho_r(x_1, \dots, x_r)$$
is assumed to occur only for $x_1, \dots, x_r \in \mathbb{R}_+$, instead of $x_1, \dots, x_r \in \mathbb{R}$, and $\rho_r$ is only defined on $\mathbb{R}_+^r$. 
\item The point process $X$ is obtained as the union of a point process $X_+$ on $\mathbb{R}_+$ whose  correlations functions are given by $\rho_r$, 
and the point process $X_-$ obtained by taking the opposites of the points of $X_+$. 
\end{itemize}
\end{proposition}
\begin{proof}
We first apply Proposition B.1 of \cite{MNN19} to the image $E_n$ of the  restriction of $X_n$ to the interval $[0, n\varepsilon/2\pi]$, 
by the map $\psi$ from $\mathbb{R}$ to $\mathbb{R}$ given by
$$\psi(x) = \int_0^{x} (1 + h(y)) dy.$$
In this way, we prove that the restriction of $X_n$ to $\mathbb{R}_+$ converges to a process $X_+$ on $\mathbb{R}_+$ whose  correlations functions are given by $\rho_r$.
Since $\Lambda_n$ is symmetric with respect to the real axis, $X_n$ converges to $X = X_+ \cup X_-$ where $X_-$ is the image of $X_+$ by $x \mapsto -x$. 
The conditions given in Proposition \ref{crit} are then proven exactly in the same way as in the proof of Propositon \ref{correlations}. 
\end{proof}

\subsection{Determinantal ensembles on the unit circle}
A particular case of application of Propositions \ref{correlations} and \ref{correlations2} is given by the following result: 

\begin{proposition} \label{determinantalcircle}
Let us assume  that the point process $\Lambda_n$ on the unit circle satisfies the following conditions:
\begin{itemize}
\item The exist $C > 0$, $\gamma \in (0,1)$, such that for all $n \geq 1$, the one-point correlation function $\rho^{(n)}_1$ is defined everywhere on the unit circle, and one has for all $\theta \in (-\pi, \pi]$, 
 $$\left| \int_0^{\theta} \left( \rho^{(n)}_1 (e^{i \tau}) - n \right) \frac{d \tau}{2 \pi} \right| \leq C  (1 + (n | \theta|)^{\gamma} ).$$
 \item There exists $\varepsilon> 0$, and a function $h$ from $\mathbb{R}$ to $\mathbb{R}_+$, locally integrable, such that 
 for all $x \in [- n \varepsilon/2 \pi, n \varepsilon/2 \pi]$, 
 $$\rho^{(n)}_1 (e^{2 i \pi x/n}) \leq n h(x).$$ 
  \item There exists a finite family of arcs $I_1, \dots, I_k$ covering the unit circle, such that for all $p \in \{1, 2, \dots, k\}$, $n \geq 1$, the following holds: there exists a kernel $K^{(n)}_p$, given by a measurable function from $I_p^2$ to $\mathbb{C}$,  
 such that $K^{(n)}_p(z_1, z_2) = \overline{K^{(n)}_p(z_2,z_1)}$ for all $z_1,  z_2 \in I_p$ (i.e. $K$ is Hermitian), 
for all $z_1, z_2, \dots, z_r \in I_p$, 
 $$\rho^{(r)}_n (z_1, \dots, z_r) := \operatorname{det} \left( (K^{(n)}_p(z_a, z_b))_{1 \leq a, b \leq r} \right)$$
 is non-negative, and the restriction of $\Lambda_n$ to $I_p$ admits $\rho^{(r)}_n$ as its $r$-point correlations function (i.e. $\Lambda_n \cap I_p$ is a determinantal point process of kernel $K^{(n)}_p$).
 \item Either $1$ is in the interior of  the arc $I_1$ (case 1), or $I_1 = e^{i[0, \theta]}$ for some $\theta \in (0, \pi]$ and for all $n \geq 1$, the points of $\Lambda_n$ are almost surely symmetrically distributed with respect to the real axis (case 2).
\item There exists a function $K$, from $\mathbb{R}^2$ to $\mathbb{C}$ in case 1, and from $\mathbb{R}_+^2$ to $\mathbb{C}$
in case 2, such that 
$$n^{-1} K^{(n)}_1 (e^{2 i \pi x/n}, e^{2 i \pi y/n}) \underset{ n \rightarrow \infty}{\longrightarrow} K(x,y)$$
for all $x, y \in \mathbb{R}$ in case 1,  and for all $x, y \in \mathbb{R}_+$ in case 2. 
\end{itemize}
Then, in case 1, there exists a determinantal point process $X$ on the real line with kernel $K$, and in case 2, there exists 
a determinantal point process $X_+$ on $\mathbb{R}_+$ with kernel $K$. If, in case 2, we define $X$ as the union of the points of $X_+$ and their opposites, 
then  the convergence of random holomorphic functions described in 
Proposition \ref{invariantcircle} occurs. 
\end{proposition}

\begin{proof}
We have to check the assumptions of Proposition \ref{correlations} in case 1 and Proposition \ref{correlations2} in case 2. 
The existence and estimate of the one-point correlation functions are already assumed. For $p \in \{1,2, \dots, k\}$, $z_1, z_2, \dots z_r \in I_p$, $\delta \geq 0$
the matrix 
$$\left( \frac{\delta \mathds{1}_{a = b}  + K^{(n)}_p(z_a, z_b)} { \sqrt{ (\delta +  K^{(n)}_p(z_a, z_a) ) (\delta + K^{(n)}_p(z_b, z_b))}  }\right)_{1 \leq a, b \leq r}$$ 
is Hermitian, and positive semi-definite since all its principal minors are given by linear combinations, with positive coefficients, of correlation functions of order smaller than or equal to $r$. The trace of this matrix is $r$, and then the product of its eigenvalues, which are non-negative, is at most one, by inequality of arithmetic and geometric means. 
Hence, 
$$\operatorname{det} \left( \left( \frac{\delta \mathds{1}_{a = b}  + K^{(n)}_p(z_a, z_b)} { \sqrt{ (\delta +  K^{(n)}_p(z_a, z_a) ) (\delta + K^{(n)}_p(z_b, z_b))}  }\right)_{1 \leq a, b \leq r} \right) \leq 1,$$
$$\operatorname{det} \left( \left( \delta \mathds{1}_{a = b}  + K^{(n)}_p(z_a, z_b)  \right)_{1 \leq a, b \leq r} \right) \leq \prod_{j=1}^r (\delta +  K^{(n)}_p(z_j, z_j))$$
and by letting $\delta \rightarrow 0$, 
$$\rho^{(r)}_n (z_1, \dots, z_r) \leq  \prod_{j=1}^r \rho^{(1)}_n (z_j).$$
For $r = 2$, this gives the second item of Proposition \ref{correlations}. The existence of the $r$-point correlations function on $I_1^r$ implies its existence 
on $J_{\varepsilon}^r$ for $\varepsilon > 0$ small enough, in both cases 1 and 2. 
For $\varepsilon> 0$ small enough, and $x_1, \dots, x_r \in [- n \varepsilon/2 \pi, n \varepsilon/2 \pi]$ in case 1, 
$x_1, \dots, x_r \in [0, n \varepsilon/2 \pi]$ in case 2, we have 
 $e^{2 i \pi x_1}, \dots, e^{2 i \pi x_r} \in I_1$, and 
 $$\rho^{(r)}_n (e^{2 i \pi x_1}, \dots, e^{2 i \pi x_r}) \leq  \prod_{j=1}^r \rho^{(1)}_n (e^{2 i \pi x_j}) \leq n^r h(x_1) \dots h(x_r)$$
 by the inequality proven just above and the assumptions made on $ \rho^{(1)}_n$. 
 For $x_1, \dots, x_r \in \mathbb{R}$ in case 1, $x_1, \dots, x_r \in \mathbb{R}_+$ in case 2, we have 
 $$n^{-r} \rho^{(r)}_n (e^{2 i \pi x_1}, \dots, e^{2 i \pi x_r})
 \underset{n \rightarrow \infty}{\longrightarrow} \rho_r(x_1, \dots, x_r)$$
 where 
 $$\rho_r (x_1, \dots, x_r) = \operatorname{det} \left( (K(x_a, x_b))_{1 \leq a, b \leq r} \right).$$
 We have now all the assumptions needed in order to apply Proposition \ref{correlations} or Proposition \ref{correlations2}. 

\end{proof}
Proposition \ref{determinantalcircle} immediately applies to the Circular Unitary Ensemble. 
One can now consider the eigenvalues of a uniform matrix on the  special orthogonal group $SO(n)$ ($n \geq 2$), multiplied a given element $u$ of the unit circle.  

This example is already studied in  \cite{CHNNR19} when $u \notin \{-1,1\}$. 
The points are symmetrically distributed with respect to the straight line between $-u$ and $u$, which divides the circle into two arcs $I$ and $I'$ of length $\pi$, namely
$I = u e^{ i [ 0, \pi]}$ and $I' = u e^{i [-\pi, 0]}$. If we remove the point $u$ for $n$ even, we get
that on $I$, the point process we consider is determinantal with kernel
$$K^{(n)}(e^{i \theta_1}, e^{i \theta_2} ) = \frac{\sin((n-1) (\theta_2 - \theta_1)/2)}{ \sin( (\theta_2 - \theta_1)/2)} 
+  (-1)^n \, \frac{\sin((n-1) (\theta_2 + \theta_1 - 2 \psi)/2)}{ \sin( (\theta_2 + \theta_1 -2 \psi)/2)} $$
where $\psi \in (-\pi, \pi]$ is the argument of $u$, and the arguments $\theta_1, \theta_2$ are taken in the interval $[\psi, \psi + \pi]$. 
In particular, the one-point correlation function is well-defined, and given by
$$\rho^{(n)}_1 (e^{i \theta}) = 
n- 1 + (-1)^n \frac{\sin((n-1) (\theta -  \psi))}{ \sin( \theta - \psi)} $$
for $\theta \in [\psi, \psi +\pi]$, and then for all $\theta \in \mathbb{R}$ by symmetry of the setting. 
We have 
\begin{align*}
\left| \rho^{(n)}_1 (e^{ i \theta}) - n \right| & \leq 1  +  \min( n -1, |\sin (\theta - \psi)|^{-1})
\\ & \leq 1 + \min(n, (\pi/2) (\operatorname{dist} (\theta, \psi + \pi \mathbb{Z}))^{-1})
\\ & = \mathcal{O} \left( 1 + \mathds{1}_{\theta \in \psi + \pi \mathbb{Z} + [-1/n, 1/n]} 
+ (\operatorname{dist} (\theta, \psi + \pi \mathbb{Z})) + (1/n))^{-1} \right).
\end{align*}
If $\theta \in (\pi, \pi]$, we have $\theta - \psi \in (- 2 \pi, 2 \pi)$, and then the minimum of $|\theta - \psi - \pi k|$ for $k  \in \mathbb{Z}$ is achieved 
for $k \in \{-2,-1,0, 1, 2 \}$. Hence, 
$$\left| \rho^{(n)}_1 (e^{ i \theta}) - n \right|  = \mathcal{O} \left( 1 + \sum_{k= -2}^2 
\left( \mathds{1}_{\theta \in \psi + \pi k + [-1/n, 1/n]}  + (|\theta - \psi - k \pi | + (1/n))^{-1} \right) \right).$$
The integral of $x \mapsto (|x| + (1/n))^{-1}$ on an interval of length $y$ is at most
$$\int_{-y/2}^{y/2}   (|x| + (1/n))^{-1} dx = 2 \int_0^{y/2} (x + (1/n))^{-1} dx = 2 \log (1 + n y/2) \leq 2 \log (1+ny),$$
and then for $\tau \in (-\pi, \pi]$, 
$$\left| \int_0^{\theta} \left(  \rho^{(n)}_1 (e^{ i \tau}) - n  \right) \,   \frac{d \tau}{2 \pi} \right|
= \mathcal{O}( \log (2 + n |\theta|).$$
Since $n^{-1} \rho^{(n)}_1$ is uniformly bounded, we have all the estimates on the one-point correlation functions  which are needed in Proposition \ref{determinantalcircle}.
The condition on the determinantal structure of the process is satisfied by taking $I_1 = I$, $I_2 = I'$ for $\psi \in (-\pi,0]$ and $I_1 = I'$, $I_2 = I$ for $\psi \in (0, \pi]$: 
for $\psi \notin \{0, \pi\}$, $1$ in inside the arc $I_1$, for $\psi \in \{0, \pi\}$, $I_1$ is the arc $e^{i [0, \pi]}$, and the points of $\Lambda_n$ are symmetrically distributed with respect to the real axis. 
It only remains to prove the pointwise convergence of the kernel. 
For $u \notin \{-1,1\}$ and then $\psi \notin \{0, \pi\}$, we have for all $x, y \in \mathbb{R}$, 
$$ K^{(n)}_1(e^{2 i \pi x/n}, e^{2 i \pi y/n} )
= \frac{\sin((n-1) (\theta_2 - \theta_1)/2)}{ \sin( (\theta_2 - \theta_1)/2)} 
 + \mathcal{O} ( |\sin(\psi)|^{-1})$$
 if $n$ is large enough.
 Hence, $n^{-1} K^{(n)}_1(e^{2 i \pi x/n}, e^{2 i \pi y/n} )$ converges pointwise to the sine-kernel, and then Proposition \ref{determinantalcircle} applies, with $X$ equal to 
 the determinantal sine-kernel process. 
 In other words, we recover the result proven in \cite{CHNNR19}. Notice that the fact that we have removed the point $u \neq 1$ in $\Lambda_n$ has no influence on the convergence of the renormalized characteristic polynomial, since it only removes a factor of  $ (1- u^{-1} e^{2 i \pi z/n})/(1-u^{-1})$, which tends to $1$ uniformly on compact sets. 
 
 If $u =  1$, it is necessary to remove $1$ from the point process in order to avoid that $Z_n(1) = 0$. 
 In this case, we have $\psi = 0$ and for $x, y \in \mathbb{R}_+$, 
 $$ n^{-1} K^{(n)}_1(e^{2 i \pi x/n}, e^{2 i \pi y/n} ) \underset{n \; \operatorname{even}, \; n  \rightarrow \infty}{\longrightarrow} 
 \frac{\sin((\pi (x-y))}{ \pi(x-y)}  + \frac{\sin((\pi (x+y))}{ \pi(x+y)},$$
  $$ n^{-1} K^{(n)}_1(e^{2 i \pi x/n}, e^{2 i \pi y/n} ) \underset{n \; \operatorname{odd}, \; n  \rightarrow \infty}{\longrightarrow} 
 \frac{\sin((\pi (x-y))}{ \pi(x-y)}  - \frac{\sin((\pi (x+y))}{ \pi(x+y)},$$
 if $n$ is large enough. 
Hence, we have convergence of the renormalized characteristic polynomial, after restricting to even or odd dimensions. The limiting point process $X_+$ is determinantal, and 
its kernel is given by one of the two limits given just above. 

If $u = -1$, and then $\psi = \pi$, we have that $I = I_2$, and then for $\theta_1, \theta_2 \in [\pi, 2 \pi]$, 
$$K^{(n)}_2(e^{i \theta_1}, e^{i \theta_2} ) = \frac{\sin((n-1) (\theta_2 - \theta_1)/2)}{ \sin( (\theta_2 - \theta_1)/2)} 
+  (-1)^n \, \frac{\sin((n-1) (\theta_2 + \theta_1 - 2 \pi)/2)}{ \sin( (\theta_2 + \theta_1 -2 \pi)/2)},$$
The point process we consider is symmetric with respect to the real axis, and then for $\theta_1, \theta_2 \in [0, \pi]$, 
\begin{align*}
K^{(n)}_1(e^{i \theta_1}, e^{i \theta_2} )
& = K^{(n)}_2(e^{i (2 \pi - \theta_1)}, e^{i  ( 2 \pi - \theta_2)} ))
\\ & =  \frac{\sin((n-1) (\theta_1 - \theta_2)/2)}{ \sin( (\theta_1 - \theta_2)/2)} 
+  (-1)^n \, \frac{\sin((n-1) (2 \pi - \theta_2 - \theta_1)/2)}{ \sin( (2 \pi - \theta_2 - \theta_1)/2)}
\\ & \frac{\sin((n-1) (\theta_1 - \theta_2)/2)}{ \sin( (\theta_1 - \theta_2)/2)} 
+  \, \frac{\sin((n-1) (\theta_1 + \theta_1)/2)}{ \sin( (\theta_2 + \theta_1)/2)}
\end{align*} 
In this case, we have convergence of the holomorphic function without restricting to even or odd dimensions. The limiting point process $X_+$ is 
determinantal with kernel  
$$K(x,y) = \frac{\sin((\pi (x-y))}{ \pi(x-y)}  + \frac{\sin((\pi (x+y))}{ \pi(x+y)}$$
which corresponds to the same limit as what we obtain for $u = 1$ in even dimensions.

The same discussion can be done when the orthogonal group is replaced by the symplectic group $\operatorname{Sp} (2n)$. In the case, the limiting function is the same as what we
get for the orthogonal group of odd dimensions, as we can check by comparing the corresponding kernels: see for example \cite{CHNNR19}.
\color{black}
\subsection{Gaussian ensembles}

For $\beta > 0$, we consider the Gaussian $\beta$ Ensemble (G$\beta$E), given by $n$ random points $\lambda^{(n)}_1, \dots, \lambda^{(n)}_n$
on the real line, whose joint density, with respect to the Lebesgue measure, is
$$Z_{n, \beta}^{-1} \, e^{- ( \beta/4) \sum_{1 \leq j \leq n} (\lambda^{(n)}_j)^2} \, \prod_{1 \leq j < k \leq n} |\lambda^{(n)}_j - \lambda^{(n)}_k|^{\beta}$$
for a suitable normalization constant $Z_{n, \beta}$. 
For $\beta \in \{1,2,4\}$, we we get the spectrum of the Gaussian Orthogonal
($\beta = 1$), Unitary ($\beta = 2$) and Symplectic ($\beta =4 $) Ensembles. 
The empirical measure associated to the points $\lambda^{(n)}_1/ \sqrt{n}, \lambda^{(n)}_2/\sqrt{n}, \dots, \lambda^{(n)}_n/\sqrt{n}$ tends in law to the semi-circular distribution, 
which is supported by the interval $[-2, 2]$. 
We now rescale the G$\beta$E around a point in the bulk of the spectrum: we fix $E \in (-2,2) $, and we 
define the point process $X_n$ by 
$$X_n := \left\{ \frac{ (\lambda^{(n)}_j - E \sqrt{n}) \sqrt{n (4 - E^2)}}{2 \pi}, 1 \leq j \leq n\right\}. $$
This rescaling is chosen in order to have an average spacing close to one, for the points of $X_n$ around zero. 
 Valk\'o and Vir\'ag have proven in \cite{VV09} that $X_n$ converges in law to $X$, where $X$ is a $\operatorname{Sine}_{\beta}$ point process, 
 normalized in such a way that the average spacing between the points is equal to one. 
 We choose the measures $M_{0,n}$ and $M_0$ as follows: 
 \begin{itemize}
 \item The measure $M_{0,n}$ is the restriction to $\mathbb{R} \backslash [-1,1]$ of the suitably rescaled semi-circular distribution: more precisely, $M_{0,n}$ has 
 a density
 $$x \mapsto (4 - E^2)^{-1/2} \sqrt{\max(0,4 - (E  + 2 \pi x n^{-1} (4 - E^2)^{-1/2})^2)} \mathds{1}_{x \not\in [-1,1]}.$$
 \item The measure $M_0$ is the Lebesgue measure on $\mathbb{R} \backslash [-1,1]$. 
 \end{itemize}
It is clear that $M_{0,n}([-1,1]) = 0$  and that $M_{0,n}$ converges vaguely to $M_0$ when $n \rightarrow \infty$. 
The densities of $M_0$ and $M_{0,n}$ are uniformly bounded by $2/(\sqrt{4 - E^2})$, which ensures that the conditions in Proposition \ref{crit} relative 
to $M_{0,n}([x,y])$ and $M_0 ([x,y])$ are satisfied, for $\nu > 0$ arbitrarily small. 
In \cite{NV21}, Najnudel and Vir\'ag prove, after suitable rescaling of the points, that 
$$\mathbb{E} [ (M_{X_n}(x) -  M_{0,n} (x))^2 ] \leq C \log (2 + \min(n, |x|))$$
and 
$$\mathbb{E} [ (M_{X}(x) -  M_{0} (x))^2 ] \leq C \log (2 + |x|)$$
where $C > 0$ depends only on $E$ and $\beta$. 
These estimate imply the conditions given in Proposition \ref{crit}, for $\alpha> 0$ arbitrarily small. 

Let us now consider the characteristic polynomial associated to the G$\beta$E: 
$$Z_n(z) := \prod_{j=1}^n (z - \lambda^{(n)}_j)$$
We observe that for $s \in \mathbb{C}$,
\begin{align*} &  \frac{Z_{n}(E \sqrt{n} + 2 \pi s (n(4 - E^2))^{-1/2})}{Z_n (E \sqrt{n})}
\\ & = \prod_{j=1}^n \frac{ E \sqrt{n} - \lambda^{(n)}_k + 2 \pi s (n(4 - E^2))^{-1/2}}{ E \sqrt{n} - \lambda^{(n)}_k} 
\\ & = \prod_{j=1}^n \left( 1 + \frac{s}{ (2 \pi)^{-1}  ( E \sqrt{n} - \lambda^{(n)}_k)(n(4 - E^2))^{1/2}} \right)
\\ & = \prod_{\rho \in X_n} \left(1 - \frac{s}{\rho} \right).
\end{align*}
Moreover, we have
\begin{align*}
& \int_{-\infty}^{\infty} \frac{d M_{0,n}(\rho)}{\rho} 
\\ & =(4 - E^2)^{-1/2}   \int_{\mathbb{R} \backslash [-1,1]}  \sqrt{\max(0,4 - (E  + 2 \pi x n^{-1} (4 - E^2)^{-1/2})^2)} \frac{dx}{x}
\\ & = (4 - E^2)^{-1/2}  \int_{\mathbb{R}} \mathds{1}_{|y| > 2 \pi n^{-1} (4 - E^2)^{-1/2}}  \sqrt{\max(0,4 - (E  + y)^2)} \frac{dy}{y}
\\ & \underset{n \rightarrow \infty}{\longrightarrow} (4 - E^2)^{-1/2}  \int_{-2}^2 \frac{\sqrt{4 - t^2}}{t - E} dt,
\end{align*}
where the last integral is understood in the sense of a principal value, and is equal to $-\pi E$.
Hence, 
$$\int_{-\infty}^{\infty} \frac{d M_{0,n}(\rho)}{\rho}  \underset{n \rightarrow \infty}{\longrightarrow} - \pi E  (4 - E^2)^{-1/2}.$$
From Theorem \ref{main}, we then deduce 
$$ \frac{Z_{n}(E \sqrt{n} + 2 \pi s (n(4 - E^2))^{-1/2})}{Z_n (E \sqrt{n})}
\underset{n \rightarrow \infty}{\longrightarrow} \underset{A \rightarrow \infty}{\lim} e^{\frac{s \pi E}{\sqrt{4 - E^2}}} \, \prod_{\rho \in X \cap [-A,A]} \left(1 - \frac{s} {\rho} \right)$$
in distribution, for the topology of the uniform convergence on compact sets: recall that $X$ is a $\operatorname{Sine}_{\beta}$ point process. 
For $\beta = 2$, this result has been proven in \cite{CHNNR19} in a different way. 

We can also consider the G$\beta$E near the edge of the spectrum. In this case, we consider 
$$X_n =  \left\{ \frac{n^{1/6}  (\lambda^{(n)}_j - 2 \sqrt{n})}{2 \pi} , 1 \leq j \leq n\right\},$$
the measure $M_{0,n}$ with density
$$x \mapsto  n^{1/3}  \sqrt{ \max(0, 4 -  (2 + 2 \pi x n^{-2/3})^2) } \mathds{1}_{x \leq -1} $$
with respect to the Lebesgue measure, and the measure $M_0$ with density
$$x \mapsto \sqrt{ - 8 \pi x} \mathds{1}_{x \leq -1}.$$
It is clear that $M_{0,n}$ and $M_0$ are supported on $\mathbb{R} \backslash [-1,1]$ and that $M_{0,n}$ converges vaguely to $M_0$ when $n \rightarrow \infty$. 
For $x > -1$, the density of the law of $M_{0,n}$ vanishes, and  for $x > -1$, it is equal to  
\begin{align*} 
n^{1/3}  \sqrt{ \max(0, 4 -  (2 + 2 \pi x n^{-2/3})^2) }
& = n^{1/3}  \sqrt{ \max(0,  - 8 \pi x n^{-2/3} - 4 \pi^2 x^2 n^{-4/3}) }
\\ & \leq  n^{1/3}  \sqrt{ \max(0,  - 8 \pi x n^{-2/3}) }
\\ & =   \sqrt{ - 8 \pi x} 
\end{align*}
We deduce that the condition of Proposition \ref{crit} on $M_{0,n} ([x,y])$ and $M_{0} ([x,y])$ is satisfied, for $\nu = 1/2$.

In the proof of  Proposition 10 in \cite{NV21},  equation (13) shows, after taking account of the scaling, that for some $C > 0$ depending only on $\beta$, 
$$\mathbb{E} [ (M_{X_n} ( (x, \infty)) - M_{0,n} ((x, \infty)) )^2]  \leq C \log(2 + n_0),$$
for $$n_0 = \max \left(1, n - \frac{\mu^2}{4} - \frac{1}{2} \right), $$
with 
$$\mu = 2 \sqrt{n} + 2 \pi x n^{-1/6}.$$
In \cite{NV21}, one sees that for $\mu \geq 0$, the bound proportional to $\log (2 +n_0)$ comes from equation (5) in Proposition 10, which is deduced from equations (17) and (18). 
In the proof of equation (17), the bound we really get is proportional to
$$\pi^2 + \sum_{n_0 - \ell + 1 \leq k \leq n_0, n_0 - k \in \mathbb{Z}} \frac{1}{k} = \mathcal{O} \left( 1 + \log  \left( \frac{n_0+1}{n_0 - \ell +1} \right) \right)$$
where 
$$\ell := \max \left(0, \lceil n_0 - \mu^{2/3} - 1\rceil \right) \in [0, n_0].$$
Hence, we have a bound of the form 
$$\mathbb{E} [ (M_{X_n} ( (x, \infty)) - M_{0,n} ((x, \infty)) )^2] \leq C   \left( 1 + \log  \left( \frac{n_0+1}{n_0 - \ell +1} \right) \right),$$
If $x \geq 0$, we have $\mu \geq 2 \sqrt{n}$, $n_0 = 1$, $\ell = 0$, and then 
$$\mathbb{E} [ (M_{X_n} ( (x, \infty)) - M_{0,n} ((x, \infty)) )^2]  = \mathcal{O}(1).$$
If $-n^{1/6} \leq x \leq 0$, we have 
$$n - \frac{\mu^2}{4} - \frac{1}{2} = n - \left( n + 2 \pi x n^{1/3} + \pi^2 x^2 n^{-1/3} \right) - \frac{1}{2} = - 2 \pi x n^{1/3} + \mathcal{O}(1),$$
$$n_0   =  - 2 \pi x n^{1/3}  + \mathcal{O}(1),$$
$$ \mu = 2 \sqrt{n} + 2 \pi x n^{-1/6} \geq 2 \sqrt{n} - 2 \pi \geq 0$$
as soon as $n \geq 10$,  and in this case, 
$$\mu^{2/3}  = 2^{2/3} n^{1/3} \left( 1 + \pi x  n^{-2/3} \right)^{2/3} 
= 2^{2/3} n^{1/3} \left( 1 + \mathcal{O}(n^{-1/2} ) \right).$$
If $-1/10 \leq x \leq 0$, we get
$$n_0 - \mu^{2/3}  \leq (\pi/5 - 2^{2/3} ) n^{1/3} + \mathcal{O} (1) \leq 0$$
for $n$ large enough, and then $\ell = 0$. 
We deduce 
$$\mathbb{E} [ (M_{X_n} ( (x, \infty)) - M_{0,n} ((x, \infty)) )^2]  = \mathcal{O}(1).$$
If $-n^{1/6} \leq x \leq -1/10$, 
we get 
$$\ell = n^{1/3} \, \max((- 2 \pi x - 2^{2/3}), 0) + \mathcal{O}(1),$$
$$n_0 - \ell = n^{1/3} \min( 2^{2/3}, - 2 \pi x) + \mathcal{O}(1),$$
and 
\begin{align*}
& \mathbb{E} [ (M_{X_n} ( (x, \infty)) - M_{0,n} ((x, \infty)) )^2] 
\\ &  = \mathcal{O}\left( 1 + \log \left(   \frac{- 2 \pi x n^{1/3}  + \mathcal{O}(1)}{ n^{1/3} \min( 2^{2/3}, - 2 \pi x) + \mathcal{O}(1)} \right) \right).
\end{align*}
Since $-x \geq 1/10$ in this case, for $n$ large enough, the numerator is smaller than $7 |x| n^{1/3}$ and  the denominator is larger than $0.6 n^{1/3}$. 
We deduce
$$ \mathbb{E} [ (M_{X_n} ( (x, \infty)) - M_{0,n} ((x, \infty)) )^2]  = \mathcal{O} ( 1+  \log  |x|)$$
for $n$ large enough. 
Finally, for $x \leq -n^{1/6}$, Theorem 7 in \cite{NV21} gives 
$$\mathbb{E} [ (M_{X_n} ( (x, \infty)) - M_{0,n} ((x, \infty)) )^2] = \mathcal{O} (\log (2 + n)) = \mathcal{O} ( 1  + \log |x| ).$$
Hence, for all $x \in \mathbb{R}$, we have, 
$$\mathbb{E} [ (M_{X_n} ( (x, \infty)) - M_{0,n} ((x, \infty)) )^2]  = \mathcal{O} \left( \log ( 2 + \max(0, -x) ) \right),$$
in particular
$$\mathbb{E} [ (M_{X_n} ( (0, \infty)) - M_{0,n} ((0, \infty)) )^2]  = \mathcal{O}(1)$$
and then 
$$\mathbb{E} [ (M_{X_n} (x) - M_{0,n} (x ))^2] =   \mathcal{O} \left( \log ( 2 + \max(0, -x) ) \right) = \mathcal{O} \left( \log (2 + |x|) \right)$$
In \cite{RVV06}, it is proven that $X_n$ converges to a point process $X$, given by the eigenvalues of a random operator, called {\it stochastic Airy operator}.  
For $a < b < c< d$, let $\varphi$ be a continuous  function from $\mathbb{R}$ to $[0,1]$, equal to $1$ on $[b,c]$ and to $0$ outside $[a,d]$. 
From the convergence of $X_n$ towards $X$, we deduce, for $A > 0$, and $j \in \{1,2\}$:
\begin{align*}
\mathbb{E} [  \min(A, ( M_{X} ([b,c]))^j ) ] &  \leq \mathbb{E} \left[ \max\left(0,  \min\left(A, \left(\int_{\mathbb{R}} \varphi  dM_{X}  \right)^j \right)  \right) \right] 
\\ & = \underset{n \rightarrow \infty}{\lim} \,  \mathbb{E} \left[ \max  \left(0,  \min\left(A, \left(\int_{\mathbb{R}} \varphi dM_{X_n}  \right)^j \right)  \right) \right]  
\\ & \leq  \underset{n \rightarrow \infty}{\lim \inf} \,  \mathbb{E} [ ( M_{X_n} ([a,d]))^j ] 
\end{align*}
Hence, letting $A \rightarrow \infty$, we get
$$\mathbb{E} [ ( M_{X} ([b,c]))^j  ]  \leq  \underset{n \rightarrow \infty}{\lim \inf} \,  \mathbb{E} [ ( M_{X_n} ([a,d]))^j ].$$
Similarly, 
$$\mathbb{E} [ ( M_{X} ([a,d]))^j  ]  \geq  \underset{n \rightarrow \infty}{\lim \sup} \,  \mathbb{E} [ ( M_{X_n} ([b,c]))^j ].$$
The bound on the $L^2$ norm of $M_{X_n} (x) - M_{0,n} (x)$ shows that for any bounded interval $I$, and all $n \geq 1$, 
$\mathbb{E} [ ( M_{X_n} (I))^2]$ is finite and bounded with respect to $n$. Hence, 
$\mathbb{E} [ (M_X(I))^2 ] < \infty$  for all bounded intervals $I$. 
By dominated convergence, 
$$\mathbb{E} [ ( M_{X} ((a,d)))^j  ] = \underset{ b \rightarrow a, c \rightarrow d}{\lim}  \,  \mathbb{E} [ ( M_{X} ([b,c]))^j  ]  $$
and 
$$\mathbb{E} [ ( M_{X} ([b,c])^j  ] = \underset{ a \rightarrow b, d \rightarrow c}{\lim}   \, \mathbb{E} [ ( M_{X} ([a,d]))^j  ].$$
We deduce, for all $a< d$,  
$$\mathbb{E} [ ( M_{X} ( (a,d)))^j  ]  \leq  \underset{n \rightarrow \infty}{\lim \inf} \,  \mathbb{E} [ ( M_{X_n} ([a,d]))^j ] 
$$ $$\leq  \underset{n \rightarrow \infty}{\lim \sup} \,  \mathbb{E} [ ( M_{X_n} ([a,d]))^j ] 
\leq \mathbb{E} [ ( M_{X} ( [a,d]))^j  ]$$

Hence, 
$$\mathbb{E} [ (M_{X} ( (a,d)))^2 - 2 M_0( [a,d]) M_X ([a,d]) +  ( M_0( [a,d]) )^2 ] 
$$ $$\leq  \underset{n \rightarrow \infty}{\lim \inf} \,  \mathbb{E} [ ( M_{X_n} ([a,d]) - M_{0,n} ([a,d]) )^2] 
= \mathcal{O} ( \log (2 + |a| + |d|) ),
$$
which implies that  
$$\mathbb{E} [ ( M_{X} ([a,d]) - M_0 ([a,d]) )^2 ] = \mathcal{O} ( \log (2 + |a| + |d|) )$$
as soon as $\mathbb{P} [ a \in X] = \mathbb{P} [ d \in X] = 0$. 
For all $a < d$, let $(a_m)_{m \geq 1}$ be a sequence increasing to $a$ and $(d_m)_{m \geq 1}$ a sequence decreasing to $d$, such that 
$a_1 > a-1$, $d_1 < d+1$, and $X$ has almost surely no point at $a_m$ or $d_m$ for all $m \geq 1$. 
Finding such a sequence is possible, since for all $\varepsilon >0$, and for any finite subset $J \in [a-1, d+1]$ such that 
$\mathbb{P} [ t \in X] > \varepsilon$ for all $t \in J$, 
$$\varepsilon \operatorname{Card}(J)
\leq \sum_{t \in J} \mathbb{P} [ t \in X]=  
\mathbb{E} [M_{X} (J) ] 
\leq \mathbb{E} [ M_{X} ([a-1,d+1])],$$
and then
$$\operatorname{Card}(J) \leq  \varepsilon^{-1} \mathbb{E} [ M_{X} ([a-1,d+1])] < \infty$$
which implies that there are finitely many points $t \in [a-1, d+1]$ with  $\mathbb{P} [ t \in X] > \varepsilon$ for all $\varepsilon > 0$, and then 
at most countably many points $t$ such that $\mathbb{P} [ t \in X] > 0$. 
We  get, for all $m \geq 1$, 
\begin{align*} \mathbb{E} [ ( M_{X} ([a_m,d_m]) - M_0 ([a_m,d_m]) )^2 ] & = \mathcal{O} ( \log (2 + |a_m| + |d_m|) ) \\ & = \mathcal{O} ( \log (2 + |a| + |d|) ) 
\end{align*}
and then by dominated convergence, 
$$\mathbb{E} [ ( M_{X} ([a,d]) - M_0 ([a,d]) )^2 ] = \mathcal{O} ( \log (2 + |a| + |d|) )$$
for all $a < d$. 
We now have all the estimates needed in order to apply Proposition \ref{crit}, for $\alpha > 0$ arbitrarily small and $\nu = 1/2$. Therefore, the conclusion of Theorem \ref{main} holds. 
We have, for $n$ large enough, 
\begin{align*}
\int_{-\infty}^{\infty} \frac{d M_{0,n}(\rho)}{\rho} & = \int_{-\infty}^{-1}  n^{1/3}  \sqrt{ \max(0, 4 -  (2 + 2 \pi x n^{-2/3})^2) } \frac{dx}{x} 
\\ & =  J_n 
+  \int_{-\infty}^{0}  n^{1/3}  \sqrt{ \max(0, 4 -  (2 + 2 \pi x n^{-2/3})^2) } \frac{dx}{x} 
\\ & = J_n +  \int_{-\infty}^{0} n^{1/3} \sqrt{  \max(0, 4 -  (2 + y)^2) } \frac{dy}{y} 
\\ & = J_n  + n^{1/3}  \int_{-2}^{2} \sqrt{  4 -  t^2 }  \frac{dt}{t - 2}
\\ & = J_n - 2 \pi n^{1/3}
\end{align*} 
where 
$$J_n =  - \int_{-1}^{0}  n^{1/3}  \sqrt{ \max(0, 4 -  (2 + 2 \pi x n^{-2/3})^2) } \frac{dx}{x}$$
tends to 
$$J_{\infty}  = - \int_{-1}^{0}  \sqrt{ \max(0, - 8 \pi x) } \frac{dx}{x}. 
$$
when $n \rightarrow \infty$. 
On the other hand, for $A > 1$, 
\begin{align*}
\int_{-A}^{A} \frac{d M_0 (\rho)}{d \rho} 
& = J_{\infty}  + \int_{-A}^{0} \sqrt{-8 \pi x}  \, \frac{dx}{x} 
\\ & = J_{\infty} -  \sqrt{8 \pi}  \int_{0}^{A} \frac{dx}{\sqrt{x}}
\\ & = J_{\infty} - 4 \sqrt{2  \pi A}. 
\end{align*}  
Theorem \ref{main} provides the convergence 
$$e^{-2  \pi s n^{1/3}}  \prod_{\rho \in X_n} \left( 1 - \frac{s}{\rho} \right)
\underset{n \rightarrow \infty}{\longrightarrow} 
\underset{A \rightarrow \infty}{\lim} e^{- 4s \sqrt{2 \pi A}} \prod_{\rho \in X \cap [-A,\infty)} \left( 1 - \frac{s}{\rho} \right).$$
Notice that we have replaced $X \cap [-A,A]$ because $X$ is almost surely bounded from above: this last fact can be deduced from the 
bound
$$\mathbb{E} [ (M_{X_n} ( (0, \infty)) - M_{0,n} ((0, \infty)) )^2]  = \mathcal{O} \left( 1 \right)$$
proven before. 
Hence, we have the convergence in law, for the topology of uniform convergence on compact sets, 
$$e^{-2 \pi s n^{1/3}}  \frac{Z_n ( 2 \sqrt{n} + 2 \pi s  n^{-1/6})} {Z_n(2 \sqrt{n})} 
\underset{n \rightarrow \infty}{\longrightarrow} 
\underset{A \rightarrow \infty}{\lim} e^{- 4s \sqrt{2 \pi A}} \prod_{\rho \in X \cap [-A,\infty)} \left( 1 - \frac{s}{\rho} \right).$$
This corresponds to a result proven by Lambert and Paquette in \cite{LP20}. 

\begin{remark}
It may be possible get convergence of renormalized characteristic polynomials of Wigner matrices, if estimates similar to those of \cite{NV21} are satisfied. To our knowledge, 
such estimates are currently not proven, and showing them is beyond the scope of this article. 
\end{remark}

\section{Proof of Theorem \ref{stozetaeq} }
We have the Hadamard product formula
$$\zeta(s) = \frac{e^{bs}}{2 (s-1) \Gamma(1 + s/2)} \prod_{\rho \in Z} (1- s/\rho)e^{s/\rho} $$
where $Z$ is the set of non-trivial zeros  of $\zeta$, and $b$ is a universal constant (see, for example, Chapter II of  \cite{T86}). 
Hence, 
$$\zeta(s) = \frac{e^{(b+c)s}}{2 (s-1) \Gamma(1 + s/2)}  \underset{A \rightarrow \infty}{\lim} \prod_{\rho \in Z, |\Im(\rho)| \leq A} (1- s/\rho)$$
for 
$$c =  \underset{A \rightarrow \infty}{\lim}  \sum_{\rho \in Z, \,  |\Im(\rho)| \leq A}  \frac{1}{\rho} 
=  \sum_{\rho \in Z, \, \Im(\rho) > 0} \left( \frac{1}{\rho} + \frac{1}{ 1-\rho} \right)
= \sum_{\rho \in Z, \, \Im(\rho) > 0} \frac{1}{\rho( 1-\rho)},$$
the last sum being absolutely convergent. 
For $u \geq 0$, let $N(u)$ be the number of zeros of $\zeta$ with imaginary part in $(0,u]$, and for $u < 0$, let $N(u)$ be the opposite of the number 
of zeros with imaginary part in $(u,0)$. We have, for $u > 0$ sufficiently large, 
$$N(u) = \frac{u}{2 \pi} \log \left( \frac{u}{2 \pi e} \right) + \mathcal{O} (\log u),$$
as stated, for example, in Theorem 9.4. of \cite{T86}. By symmetry, we deduce that for all $u \in \mathbb{R}$, 
$$N(u) = \frac{u}{2 \pi} \log \left( \frac{|u|}{2 \pi e} \right) + \mathcal{O} (\log (2 + |u|)),$$
where $0 \log 0 := 0$. 
For $v > u > 10$, we deduce
$$N(v) - N(u) = \mathcal{O} ( (1 + v-u) \log v),$$
and for $t \geq 10$ and $A \geq 2t$,  
$$N(A + 2t ) - N(A) = \mathcal{O} ( t \log A). $$
Hence, for $A \geq 2 |s|$,
$$\prod_{\rho \in Z, \Im(\rho) \in (A, A + 2t]} \left(1 - \frac{s}{\rho} \right) 
= e^{ \mathcal{O} ( |s|  (N(A + 2t ) - N(A)) /A ) } 
= e^{ \mathcal{O} (|s| t \log A / A)} 
\underset{A \rightarrow \infty}{\longrightarrow} 1,$$
and for any $t \geq 10$, 
$$\zeta(s) = \frac{e^{(b+c)s}}{2 (s-1) \Gamma(1 + s/2)}  \underset{A \rightarrow \infty}{\lim} \prod_{\rho \in Z, \Im(\rho) \in 
[-A, A+2t]} (1- s/\rho).$$

Now, for $T \geq 1000$,  $100 \leq T/\log T \leq t \leq T$,   $\zeta(1/2 + it) \neq 0$ and $1/2+i(t + 2 \pi s/\log T) \neq 1$, we deduce
\begin{align*}
& \frac{ \zeta(1/2 + i (t + 2 \pi s/ \log T))}{\zeta (1/2 + it)} 
\\ & =  ( 1 + \varphi_{t,T} (s) ) e^{i \pi s}  \underset{A \rightarrow \infty}{\lim} \prod_{\gamma \in Z_{t,T}, 
\Im(\gamma) \in [(-A - t) (\log T)/2 \pi, (A + t) (\log T)/2 \pi] }  \left(1 - \frac{s}{\gamma} \right)  
\\ & = ( 1 + \varphi_{t,T} (s) ) e^{i \pi s}  \underset{A \rightarrow \infty}{\lim} \prod_{\gamma \in Z_{t,T}, 
\Im(\gamma) \in [-A,A] }  \left(1 - \frac{s}{\gamma} \right)  
\end{align*}
where
$$Z_{t, T} =\left\{  \left( \rho - \frac{1}{2} - it \right) \frac{\log T}{2 i \pi}, \; \rho \in Z \right\}$$
and $\varphi_{t,T}(s)$ tends to zero, uniformly in $t \in [T/\log T, T]$ and uniformly  in $s$ on any fixed compact subset of $\mathbb{C}$, when $T$ goes to infinity.

Let us more precisely evaluate how the quantity above converges when $A \rightarrow \infty$.
If  for $x > A$, $N_{A,Z_{t, T}} (x)$ denotes the number of points in $Z_{t, T}$ with real part in the interval $(A, x]$ (counted with multiplicity),  we have, 
for $1 + 2 |s| \leq A  < B$, 
$$  \prod_{\gamma \in Z_{t,T}, 
\Re(\gamma) \in (A,B] }  \left(1 - \frac{s}{\gamma} \right)  
 = \exp \left( \int_A^B \log \left(1 - \frac{s}{\gamma} \right) \, d N_{A, Z_{t, T}} (\gamma) \right).
$$
Now, 
$$N_{A, Z_{t, T}} (\gamma) 
= N \left( t + \frac{2 \pi \gamma}{\log T} \right) - N(t).$$
 Since for $u \geq v$, 
 \begin{align*}
 N(v) - N(u) &  = \frac{v}{2 \pi} \log \left( \frac{|v|}{2 \pi e} \right) - \frac{u}{2 \pi} \log \left( \frac{|u|}{2 \pi e} \right) 
 + U(v) - U(u)
 \\ & = U(v) - U(u) +  \int_u^v \frac{dw}{2 \pi} \log \left| \frac{w}{2 \pi} \right|
 \end{align*} 
 where 
 $U(u) = \mathcal{O} (\log (2 + |u|))$, we 
 have
 \begin{align*}
 & \prod_{\gamma \in Z_{t,T}, 
\Re(\gamma) \in (A,B] }  \left(1 - \frac{s}{\gamma} \right)  
\\ &  =  \exp \left( \int_A^B \log \left(1 - \frac{s}{\gamma} \right) \, \log \left| \frac{t}{2 \pi} +  \frac{\gamma}{\log T} \right| \frac{d \gamma}{\log T} \right)
\\ & \times \exp \left( \int_A^B \log \left(1 - \frac{s}{\gamma} \right) \, d \, U \left( t + \frac{2 \pi \gamma}{\log T} \right) \right)
 \end{align*} 
Since $|s| \leq A/2$ by assumption, the estimate on $U$ implies the following bound on the  term inside the second exponential:
\begin{align*} &  \left[ \log \left(1 - \frac{s}{\gamma} \right) \, U \left( t + \frac{2 \pi \gamma}{\log T} \right)  \right]_A^B - \int_A^B 
 \frac{s}{ \gamma (\gamma - s)}  U  \left( t + \frac{2 \pi \gamma}{\log T} \right) d \gamma.
 \\ & = \mathcal{O} \left( (  \log T ) \left(\frac{|s|}{B} + \frac{|s|}{A} + \int_{A}^B \frac{|s|}{\gamma (\gamma - |s|)} d \gamma \right) \right)
 \\ & + \mathcal{O} \left(
  \left( \frac{|s| \log (1 + |A|) }{|A|} +  \frac{|s| \log (1 + |B|) }{|B|} +  \int_{A}^B \frac{|s| \log (1 + \gamma)}{\gamma (\gamma - |s|)} d \gamma \right) \right)
  \\ & = \mathcal{O} \left( \frac{|s| (\log T + \log (1+A))}{A} \right).
  \end{align*}
For the first term, we get 
\begin{align*}
&  \int_A^B \log \left(1 - \frac{s}{\gamma} \right) \, \log \left| \frac{t}{2 \pi} +  \frac{\gamma}{\log T} \right| \frac{d \gamma}{\log T} 
\\ & = -s  \int_A^B  \, \log \left| \frac{t}{2 \pi} +  \frac{\gamma}{\log T} \right| \frac{d \gamma}{\gamma \log T} 
\\ & + \mathcal{O} \left(  \int_A^B \frac{|s|^2}{ \gamma^2}  \, (\log T + \log (1 + \gamma) ) \frac{d \gamma}{\log T} \right)
\\ & = -s  \int_A^B  \, \log \left| \frac{t}{2 \pi} +  \frac{\gamma}{\log T} \right| \frac{d \gamma}{\gamma \log T} 
+ \mathcal{O} \left( \frac{ |s|^2 \log (2 + A) }{A} \right)
\\ & = - \frac{s (\log B - \log A) \log (t/2 \pi)}{\log T}  - \frac{s}{\log T} \int_A^B 
 \, \log \left| 1 +  \frac{2 \pi \gamma}{t \log T} \right| \frac{d \gamma}{\gamma} 
\\ &  + \mathcal{O} \left( \frac{ |s|^2 \log (2 + A) }{A} \right)
\\ & = - \frac{s (\log B - \log A) \log (t/2 \pi)}{\log T}  - \frac{s}{\log T} \int_{2 \pi A/(t \log T)}^{2 \pi B/(t \log T)}
 \, \log \left| 1 + y \right| \frac{d y}{y} 
\\ &  + \mathcal{O} \left( \frac{ |s|^2 \log (2 + A) }{A} \right)
\end{align*}
Hence, we have 
 $$ \prod_{\gamma \in Z_{t,T}, 
\Re(\gamma) \in (A,B] }  \left(1 - \frac{s}{\gamma} \right)  
=  \exp \left(- \frac{s (\log B - \log A) \log (t/2 \pi)}{\log T}  \right.$$ $$\left. - \frac{s}{\log T} \int_{2 \pi A/(t \log T)}^{2 \pi B/(t \log T)}
 \, \log \left| 1 + y \right| \frac{d y}{y}  +  \mathcal{O} \left( \frac{ (|s| + |s|^2) ( \log T +  \log (2 + A) )}{A} \right) \right).$$
By doing a similar computation for $\Re(\gamma) \in [-B, -A)$, we deduce 
$$\prod_{\gamma \in Z_{t,T}, 
|\Re(\gamma) | \in (A,B] }  \left(1 - \frac{s}{\gamma} \right)  
=  \exp \left( - \frac{s}{\log T} \int_{2 \pi A/(t \log T)}^{2 \pi B/(t \log T)}
 \, \log \left| \frac{1 + y }{1-y} \right| \frac{d y}{y}  \right.$$ $$\left.   + \,  \mathcal{O} \left( \frac{ (|s| + |s|^2) ( \log T +  \log (2 + A)) }{A} \right) \right).$$
 Now, 
 $$\int_0^{\infty} \log \left| \frac{1 + y }{1-y} \right|   \frac{dy}{y} < \infty$$
and then
$$\prod_{\gamma \in Z_{t,T}, 
|\Re(\gamma) | \in (A,B] }  \left(1 - \frac{s}{\gamma} \right)  
=  \exp \left(  \mathcal{O} \left( \frac{ (|s| + |s|^2) ( \log T +  \log (2 + A)) }{A} + \frac{|s|}{\log T} \right) \right).$$
Letting $A = \log^2 T$ and  $B \rightarrow \infty$, we deduce 
\begin{align*}
& \frac{ \zeta(1/2 + i (t + 2 \pi s/ \log T))}{\zeta (1/2 + it)} 
\\ & = ( 1 + \varphi_{t,T} (s) ) e^{i \pi s}  \underset{B \rightarrow \infty}{\lim} \prod_{\gamma \in Z_{t,T}, 
\Re(\gamma) \in [-B,B] }  \left(1 - \frac{s}{\gamma} \right)  
\\ & =  ( 1 + \psi_{t,T} (s) )  e^{i \pi s}
  \prod_{\gamma \in Z_{t,T}, 
\Re(\gamma) \in [-\log^2 T,\log^2 T] }  \left(1 - \frac{s}{\gamma} \right), \end{align*}
where $\psi_{t,T}$  converges to zero uniformly in $t \in [T/\log T, T]$ and uniformly in $s$ on any fixed compact subset of $\mathbb{C}$, 
when $T \rightarrow \infty$.

In \cite{F74}, Fujii's main result implies the following: for  $T \geq 10$, $0 < h \leq T/2$, 
$$\int_T^{2T} (S(t +h) - S(t) )^2 dt = \mathcal{O} \left( T \log (2 + h \log T) \right),$$
where for $t \neq 0$ not equal to the imaginary part of a zero of $\zeta$, 
$$S(t) = \frac{1}{\pi} \Im \log \zeta(1/2 + it) :=  - \frac{1}{\pi} \int_{1/2}^{\infty} \Im (\zeta'(\sigma +it)/ \zeta(\sigma + it)) d\sigma$$
and for $t= 0$ or $t$ equal to the imaginary part of a zero of $\zeta$, 
$$S(t) = \lim_{\varepsilon \rightarrow 0, \varepsilon > 0}  S(t + \varepsilon).$$
We deduce that for $0 < h \leq T/3$, which implies $h \leq (T-h)/2$, 
\begin{align*}
& \int_T^{2T} (S(t -h) - S(t) )^2 dt
= \int_{T-h}^{2 T - h} (S(t ) - S(t + h) )^2 dt
\\ & \leq \int_{T-h}^{2 (T-h)}  (S(t +h) - S(t ) )^2 dt + \int_{2(T-h)}^{4 (T-h)}  (S(t +h) - S(t ) )^2 dt
\\ & = \mathcal{O} \left( (T-h) \log (2 + h \log (T-h)) +(T+h) \log (2 + h \log (T+h) \right),
\end{align*} 
and then
$$\int_T^{2T} (S(t +h) - S(t) )^2 dt = \mathcal{O} \left( T \log (2 + |h| \log T) \right)$$
for all $h \in [-T/3, T/2]$. 
Moreover, from the bound $S(u) = \mathcal{O} (\log (2 + |u|))$ (see \cite{T86}, Theorem 9.4.), we get 
$$\int_T^{2T} (S(t +h) - S(t) )^2 dt  = \mathcal{O} (T \log^2 (T + |h|)) $$
in any case, and then 
$$\int_T^{2T} (S(t +h) - S(t) )^2 dt  = \mathcal{O} (T \log^2  (4|h|)) $$
for $|h| \geq T/3$. 

 Combining the two bounds, we deduce, for all $T \geq 1000$, $h \in \mathbb{R}$, 
$$\int_T^{2T} (S(t +h) - S(t) )^2 dt  = \mathcal{O} (T \log^2 (2 +  |h| \log T)).$$
 For $0 \leq T \leq 1000$, we have
 $$\int_T^{2T} (S(t +h) - S(t) )^2 dt  \leq \int_0^{2000} (S(t +h) - S(t) )^2 dt
 = \mathcal{O} ( \log^2 ( 2 + |h|)),$$
  and then for all $T \geq 0$, 
  $$\int_T^{2T} (S(t +h) - S(t) )^2 dt  = \mathcal{O} (T \log^2 (2 +  |h| \log (2 + T))).$$
  Adding this expression for $T/2, T/4, T/8, \dots$, we deduce for all $T \geq 0$, 
   $$\int_0^{T} (S(t +h) - S(t) )^2 dt  = \mathcal{O} (T \log^2 (2 +  |h| \log (2 + T))).$$
 By \cite{T86}, Theorem 9.3., we have, for $ u \geq 1$ large enough, 
\begin{align*}
N(u) &  =  \frac{u}{2 \pi} \log \left( \frac{u}{2 \pi e} \right)  + S(u) + \mathcal{O}(1/u)
 \\ & = \int_0^u  \frac{1}{2 \pi}  \log \left( \frac{t}{2 \pi} \right) dt  + S(u) + \mathcal{O}(1/u)
 \\ & = \int_0^u  \frac{1}{2 \pi}  \log_+ \left( \frac{t}{2 \pi} \right) dt  + S(u) + \mathcal{O}(1)
 \end{align*} 
 where $\log_+ (a) := \max(0, \log a)$. 
The last estimate remains true for all $u \geq 0$, and then by symmetry, for all $u \in \mathbb{R}$, 
$$N(u) = \int_0^u  \frac{1}{2 \pi}  \log_+ \left( \frac{|t|}{2 \pi} \right) dt  + S(u) + \mathcal{O}(1).$$

Hence, 
  $$\int_0^{T} \left(N(t +h) - N(t) - \int_0^h  \frac{1}{2 \pi}  \log_+ \left( \frac{|v + t|}{2 \pi} \right) dv \right)^2 dt  $$ $$= \mathcal{O} (T \log^2 (2 +  |h| \log (2 + T))).$$
Now, assume the Riemann hypothesis and the fact that the zeros of $\zeta$ are simple. 
Let  $T \geq 1000$, let  $t$ be uniformly distributed in $[T/\log T,T]$, and let $X_T$ be the point process given by the set $Z_{t,T} \cap [- \log^2 T,  \log^2 T]$: by assumption, it has real and simple points. 
We define $M_{0,T}$ as the Lebesgue measure on $[-\log^2 T, \log^2 T] \backslash [-1,1]$, and 
$M_0$ as the Lebesgue measure on $\mathbb{R} \backslash [-1,1]$: it is immediate that $M_{0,T}$ converges vaguely to $M_0$ when $T \rightarrow \infty$. 
For $x \in [-\log^2 T, \log^2 T]$, with the notation of Theorem \ref{main}, we have
\begin{align*} 
&\left( \mathbb{E} [  (M_{X_T}( x) - M_{0,T} (x) )^2]  \right)^{1/2} 
\\ & \leq \left(\frac{1}{T- (T/\log T)} \int_0^T  \left( N\left(t + \frac{ 2  \pi x}{ \log T}  \right) - N(t) - \operatorname{sign}(x) \max( |x|-1, 0) dy   \right)^2 dt \right)^{1/2}
\\ & \leq  \left(\frac{2}{T} \int_0^T  \left( N\left(t + \frac{ 2  \pi x}{ \log T}  \right) - N(t) - \int_{0}^{2 \pi x/ \log T} \frac{  \log_+ (|t + v| / 2 \pi) }{2 \pi}   dv   \right)^2 
dt \right)^{1/2}
\\ & +  \left(\frac{2}{T} \int_0^T \left( \int_{0}^{2 \pi x/ \log T} \left( \frac{ \log_+ (|t + v| / 2 \pi) - \log T }{2 \pi}\right) dv   \right)^2 dt \right)^{1/2}
\\ & +  \left(\frac{2}{T}  \int_0^T \left(\int_{0}^{2 \pi x/ \log T} \frac{\log T }{2 \pi} \, dv - \operatorname{sign}(x) \max( |x|-1, 0)   \right)^2 dt \right)^{1/2}
\end{align*}
Taking $h = 2 \pi x / \log T$ in the previous estimate, we deduce that the first term is dominated by 
$$\log ( 2 + 2 \pi |x| \log (2 + T)/ \log T) = \mathcal{O} ( \log (2 + |x| )).$$
The third term is obviously uniformly bounded. 
For the second term, we observe that since $t \in [T/ \log T, T]$ and $x \in [-\log^2 T, \log^2 T]$, we have  
$$ \frac{T}{ \log T} - 2 \pi  \log T \leq t+v  \leq T + 2 \pi \log T.$$
Since $T \geq 1000$, we deduce that
$$\frac{T}{2 \log T} \leq t + v \leq 2T,$$
and then the second term above is dominated by
$$\left(\frac{2}{T} \int_0^T \left( \int_{0}^{2 \pi x/ \log T} ( \log \log T)  dv   \right)^2 dt \right)^{1/2} $$ $$= \mathcal{O} (|x| \log \log T/ \log T)
= \mathcal{O} ( |x| / (\log T)^{2/3}) = \mathcal{O} ( |x|^{2/3}),$$
the last estimate being due to the fact that $|x| \leq \log^2 T$. 
Hence, 
$$\mathbb{E} [  (M_{X_T}( x) - M_{0,T} (x) )^2]  = \mathcal{O} ( 1 + |x|^{4/3})$$
for $|x| \leq \log^2 T$, and then for all $x \in \mathbb{R}$, since both measures $M_{X_T}$ and $X_{0,T}$ are supported in 
$[-\log^2 T, \log^2 T]$. 
The assumptions on $X_T$, $M_0$ and $M_{0,T}$ made in Proposition \ref{crit} are then satisfied (with $\nu$ arbitrarily small and $\alpha = 2/3$). 
We deduce that for $t$ unifomly distributed on $[T/ \log T, T]$, and then also for $t$ uniform on $[0,T]$, the 
random holomorphic function
$$\left( \frac{ \zeta(1/2 + i (t + 2 \pi s/ \log T))}{\zeta (1/2 + it)} \right)_{s \in \mathbb{C}}$$
converges in law to 
 $$
 \left( e^{i \pi s} \, \lim_{A \rightarrow \infty} \, \prod_{\rho \in X \cap [-A,A]}  \left(1- \frac{s}{\rho} \right) \right)_{s \in \mathbb{C}},$$
for the topology of uniform convergence on compact sets, as soon as $X_T$, or equivalently $Z_{t,T}$, converges to a simple point process $X$, with almost surely no point at zero, and such that 
$$\mathbb{E} [ (M_{X}(x) - x)^2 ] = \mathcal{O} (1 + |x|^{2 \alpha})$$
for some $\alpha \in (0,1)$. By Theorem \ref{converse}, for any point process $X$ satisfying these estimates, convergence of holomorphic functions also implies the convergence of $Z_{t, T}$ towards $X$. 

\begin{remark}
    A conjecture by Rudnik and Sarnak \cite{RS96}, generalizing a conjecture by Montgomery \cite{M73}, states that convergence occurs with $X$ equal to a determinantal sine-kernel process. This conjecture, together with the Riemann hypothesis and the simplicity of the zeros, is then equivalent to the conjecture stated by 
 Chhaibi, Najnudel and Nikeghbali in \cite{CNN17}. This equivalence has been stated, with a detailed sketch of proof, in \cite{S18}, with some obvious parts left to the reader. Here we preferred to give our own detailed proof which also shows how the probabilistic ideas can fit in this analytic number theory setting.
\end{remark}

\section{Proof of Theorems \ref{Cauchy} and \ref{Cauchy-Riemann}} \label{cauchy}

\subsection{Proof of Theorem \ref{Cauchy}}
We use the following result: 
\begin{proposition} \label{thetacauchy}
 Let $\theta_1, \dots, \theta_n \in [0, \pi)$ and let $w_1, \dots, w_n$ be positive reals such that $\sum_{j=1}^n w_j = 1$. Then, for $\Theta$ uniformly distributed on $[0, \pi)$, $\sum_{j=1}^n w_j \cot(\Theta - \theta_j) $
follows a standard Cauchy distribution.
\end{proposition}
\begin{proof}
We can order the $\theta_j$'s and group the $w_j$'s corresponding to the angles which are equal to each other. Moreover, we can also translate the angles because of the periodicity of the cotangent. Hence, without loss of generality, we can assume
$$0 = \theta_1 < \theta_2 < \dots < \theta_n < \pi.$$
For $h \in \mathbb{R}$, we have to compute the Lebesgue measure of the set of $\theta \in [0, \pi)$ such that $\sum_{j=1}^n w_j \cot(\theta - \theta_j) > h$. Since the cotangent decreases from $\infty$ to $-\infty$ between its singularities, this set is the union of intervals $[\theta_j, \theta'_j)$ where $\theta_j < \theta'_j < \theta_{j+1}
$ (with the convention $\theta_{n+1} = \pi$) and $\sum_{j=1}^n w_j \cot(\theta - \theta_j) = h$ for all $\theta \in \{\theta'_1, \dots, \theta'_n\}$.  The equation can be rewritten as 
$$\sum_{j=1}^n w_j \frac{e^{ 2 i \theta} + e^{2 i \theta_j}}{e^{2i \theta} - e^{2i \theta_j}} = - i h, $$
i.e. 
$$\sum_{j=1}^n w_j  (e^{2i \theta} + e^{2i \theta_j}) \prod_{1 \leq k \leq n, k \neq j} (e^{2i \theta} - e^{2i \theta_j}) + i h \prod_{1 \leq j \leq n} (e^{2i \theta} - e^{2i \theta_j}) = 0.$$
This is a $n$ degree equation in $e^{2i \theta}$, with $n$ roots equal to $e^{2 i \theta'_j}$ for $1 \leq j \leq n$. The Lebesgue measure of the union of $[\theta_j, \theta'_j)$ is related to the product of the roots of the $n$ degree equation. The leading order term of the equation is $ih + \sum_{j=1}^n w_j = 1+ ih$. The constant term is 
$$ (-1)^{n-1} e^{2 i \sum_{j=1}^n \theta_j} \sum_{j=1}^n w_j   + ih (-1)^n  e^{2 i \sum_{j=1}^n \theta_j} = (-1)^n  e^{2 i \sum_{j=1}^n \theta_j} (ih - 1) $$
Hence, the product of the roots is 
$$ e^{2 i \sum_{j=1}^n \theta'_j} =   e^{2 i \sum_{j=1}^n \theta_j} (ih - 1)/ (ih+1).$$
We then have 
$$ \sum_{j=1}^n (\theta'_j - \theta_j) \equiv \frac{1}{2} \operatorname{Arg} ((ih-1)/(ih +1)) = \frac{1}{2} \operatorname{Arg} ((h + i)/(h- i)) \equiv   \operatorname{Arg} ( h + i) =   \operatorname{arccot} h$$
modulo $\pi$. The Lebesgue measure of $\theta \in [0, \pi)$ such that $\sum_{j=1}^n w_j \cot(\theta - \theta_j) > h$ is then congruent to $ \operatorname{arccot} h$ modulo $\pi$. Hence,
$$\mathbb{P} \left( \sum_{j=1}^n w_j \cot(\Theta - \theta_j) > h \right) = \frac{1}{ \pi} \operatorname{arccot} h,$$
since this probability is in $(0,1)$. Differentiating in $h$, we get the density of the standard Cauchy distribution. 
\end{proof}
Using the identity
$$\pi \cot (\pi x) = \underset{ A \rightarrow \infty}{\lim} \sum_{- A \leq k \leq A, \, k \in \mathbb{Z}} \, \frac{1}{x - k},$$
and doing the appropriate scaling, we deduce the following result: 
\begin{proposition} \label{cauchyperiodic}
 Let $M$ be a non-zero atomic measure on $\mathbb{R}$, with finitely many atoms in compact sets, and $B$-periodic for some $B > 0$. We define 
  $$K := B^{-1} M ([0, B) ).$$
  Then, the function 
  $$s \mapsto \frac{1}{K \pi} \underset{A \rightarrow \infty}{\lim} \int_{-A}^A \frac{ d M (\rho)}{ s - \rho}$$
  is well-defined on all points of $\mathbb{R}$ which are not atoms for $M$. This function is $B$-periodic and the image of the uniform distribution on $[0, B)$ by this function is the standard Cauchy distribution. 
\end{proposition}
\begin{proof}
We can write 
$$M = \sum_{j=1}^{n} KB w_j \sum_{m \in \mathbb{Z}}  \delta_{B ( m + \theta_j / \pi)}$$
for some $\theta_1, \dots, \theta_n \in [0, \pi)$, and for some $w_1, \dots, w_n > 0$ such that $\sum_{j=1}^n w_j = 1$. 
After checking that the change in truncation involved here does not modify the limit, we get that the function considered in the proposition is equal to 
$$ s \mapsto  \frac{B}{\pi} \sum_{j=1}^n w_j \underset{ A \rightarrow \infty}{\lim}  \sum_{-A/B \leq m \leq A/B, m \in \mathbb{Z} } \frac{ 1}{ s  - B \theta_j/\pi - Bm}
= \frac{B}{\pi}  \sum_{j=1}^n w_j \frac{\pi}{B}  \cot (\pi s/B -  \theta_j),$$
where $\pi s/B - \theta_j$ is not multiple of $\pi$, i.e. $s$ is not an atom of $M$.
If $s$ is uniformly distributed in $[0, B)$, the last expression is Cauchy-distributed by Proposition \ref{thetacauchy}.
 \end{proof}
We have the ingredients needed in order to prove Theorem \ref{Cauchy}.
For $R > 0$ such that $M((-R/2, R/2]) > 0$, let $M_R$ be the $R$-periodic measure which coincides with $M$ on the interval $(-R/2,R/2]$. 
We have, for $s \in \mathbb{R}$ which is not an atom of $M$ or $M_R$, and for $A > 0$,  
$$\int_{[-A,A]} \frac{d M(\rho)}{ s- \rho} -\int_{[-A,A]} \frac{d M_R(\rho)}{ s- \rho}
=\int_{[-A,A]} \frac{d N_R(\rho)}{ s- \rho}$$
 where $N_R$ is the signed measure $M- M_R$. Hence, for $s \in (-R/2,R/2)$,
\begin{align*}
& \int_{[-A,A]} \frac{d M(\rho)}{ s- \rho}  -\int_{[-A,A]}\frac{d M_R(\rho)}{ s- \rho} 
 \\ & = \frac{N_R(A)  }{s - A}  
 -\frac{N_R((-A)-)-  }{s + A}
 - \int_{[-A,  A]} \frac{N_R(\rho)  }{(s - \rho)^2} d \rho. 
\end{align*}
Here, the singularity $\rho = s$ does not appear since $N_R$ vanishes on $(-R/2,R/2]$ by construction. 
Since the distribution of $M$ is invariant by translation, we have 
$$\mathbb{E} [ M(x)] = x \, \mathbb{E} [ M(1)] = K x$$
for all $x \in \mathbb{R}$, and then 
$$\mathbb{E} [ (M(x) - Kx)^2 ] 
= \operatorname{Var} (M(x)) \leq C (1 + |x|^{2 \alpha}).  
$$
Hence, for $u \geq 2$ and for any integer $m \geq 1$,  
$$\mathbb{P} [ |M (m^u) - K m^u| 
\geq m^{u-1} ] \leq C ( 1 + m^{ 2  u \alpha} ) m^{- 2 (u-1)} $$
If $u$ is large enough, depending on $\alpha \in [0, 1)$, Borel-Cantelli lemma shows that almost surely, 
$$| M(m^u) - K m^u | < m^{u-1}$$
for all but finitely many integers $m$. Since $M$ is nondecreasing, we deduce that almost surely, 
$$| M(x) - K x| = \mathcal{O} (1 + |x|^{\beta})$$
for all $x \geq 0$, and by symmetry, for all $x \in \mathbb{R}$, where $\beta = (u-1)/u \in [0,1)$. Here, the implicit constant is random but does not depend on $x$. 
In particular, for $R \geq 1$, 
$$K_R := \frac{M((-R/2, R/2])}{R} = K  + \mathcal{O}(R^{\beta-1}).$$
 For $x \in (-R/2, R/2]$, we deduce 
\begin{align*}
|M_R (x) - K_R x |
& \leq |M(x) - Kx| + |K - K_R| | x| \,  
\\ & \leq \mathcal{O} (1 + |x|^{\beta}) + |x| \mathcal{O} (R^{\beta-1}) = \mathcal{O} (1 + |x|^{\beta}) 
 \end{align*}
 the random implicit constant being uniform in $x$ and $R$. Since $M_R (x) - K_R x$ is $R$-periodic in $x$ ($M_R((-R/2, R/2]) = K_R R$ by construction), we can discard the condition on $x$. Combining the bounds on $M$ and $M_R$, we deduce that uniformly on $R \geq 1$ and $x \in \mathbb{R}$, 
$$N_R(x) = (K - K_R) x +   \mathcal{O} (1 + |x|^{\beta}).$$
This estimates allows to write, for $s \in (-R/2, R/2)$ which is not an atom of $M$ or $M_R$, 
$$\int_{[-A,A]} \frac{d M(\rho)}{ s- \rho}  -\int_{[-A,A]}\frac{d M_R(\rho)}{ s- \rho}  
\underset{A \rightarrow \infty}{\longrightarrow} 
 \underset{A \rightarrow \infty}{\lim} \int_{\mathbb{R}} \frac{N_R(\rho)  }{(s - \rho)^2} d \rho.$$  
 Indeed, the sum of the boundary terms of the integration by parts above tends to zero when $A \rightarrow \infty$, and the limit at the right-hand side exists and equals 
 $$\int_{R/2}^{\infty} \left( \frac{N_R(\rho)}{(s- \rho)^2} + 
\frac{N_R(-\rho)}{(s+ \rho)^2} \right) d \rho,  $$
the integral being convergent from the above estimate of $N_R$. Notice that the integral can start at $R/2$ since $N_R$ vanishes on $(-R/2,R/2]$.  
Since $M_R$ is periodic, the integral on $[-A,A]$ with respect to $dM_R(\rho)$ converges when $A \rightarrow \infty$, and the following equality is meaningful and satisfied: 
$$\frac{1}{K_R \pi} \underset{A \rightarrow \infty}{\lim}\int_{[-A,A]} \frac{d M(\rho)}{ s- \rho} 
= \frac{1}{K_R \pi} \underset{A \rightarrow \infty}{\lim}\int_{[-A,A]} \frac{d M_R(\rho)}{ s- \rho}
+ \frac{1}{K_R\pi}\int_{R/2}^{\infty} \left( \frac{N_R(\rho)}{(s- \rho)^2} + 
\frac{N_R(-\rho)}{(s+ \rho)^2} \right) d \rho.$$
Now, let us take for $s$ a uniform random variable on $(-R/2,R/2)$, independent of $M$. By Proposition \ref{cauchyperiodic}, the first term of the right-hand side is  
distributed as a standard Cauchy variable. The last integral is 
$$ (K-K_R) \int_{R/2}^{\infty}
\rho \left( \frac{1}{(s-\rho)^2} - 
\frac{1}{(s+\rho)^2} \right) d \rho + \mathcal{O} \left( \int_{R/2}^{\infty} (1 + \rho^{\beta})
\left( \frac{1}{(s-\rho)^2} + 
\frac{1}{(s+\rho)^2} \right) d \rho \right). 
$$
almost surely, uniformly in $R \geq 1$ and $s \in (-R/2,R/2)$. 

The main term has a modulus at most: 
$$|K - K_R| \int_{R/2}^{\infty} 
\frac{ 4 |s| \rho^2}{ (\rho - |s|)^2 (\rho + |s|)^2} d \rho \leq 
4 |s| |K - K_R| \int_{R/2}^{\infty} \frac{d \rho}{( \rho - |s|)^2} = \frac{ 4 |s| |K - K_R|}{ R/2 - |s|}.  $$
The error term is dominated by 
$$\int_{R/2}^{\infty} \frac{\rho^{\beta}} {( \rho - |s|)^2} d \rho  \leq \frac{(R/2)^2}{(R/2 - |s|)^2} \int_{R/2}^{\infty} \frac{\rho^{\beta}} {\rho^2} d \rho
\leq \frac{(R/2)^2}{(R/2 - |s|)^2} 
 \frac{(R/2)^{\beta- 1}}{1 - \beta}. $$
 Since $|K - K_R| = \mathcal{O} (R^{\beta-1})$, we get almost surely, uniformly in $R \geq 1$ and $s \in (-R/2, R/2)$, 
 $$\frac{1}{K_R\pi}\int_{R/2}^{\infty} \left( \frac{N_R(\rho)}{(s- \rho)^2} + 
\frac{N_R(-\rho)}{(s+ \rho)^2} \right) d \rho
 = \mathcal{O} \left( \frac{R^{\beta}}{R/2 - |s|}
 + \frac{R^{\beta +1}}{(R/2 - |s|)^2} \right)
 = \mathcal{O} \left( \frac{R^{\beta +1}}{(R/2 - |s|)^2} \right).
$$
For $R \geq 10$, if $s$ is uniformly distributed on 
$(-R/2 + R/\log R, R/2 - R/\log R)$, then the random variable above converges to zero almost surely when $R \rightarrow \infty$. If $s$ is uniformly distributed on $(-R/2, R/2)$, it can be coupled with a uniform variable on $(-R/2 + R/\log R, R/2 - R/\log R)$ in such a way that the two variable are different with probability $\mathcal{O}(1/\log R)$, which is sufficient to deduce that 
$$ \int_{R/2}^{\infty} \left( \frac{N_R(\rho)}{(s- \rho)^2} + 
\frac{N_R(-\rho)}{(s+ \rho)^2} \right) d \rho \underset{R \rightarrow \infty}{\longrightarrow} 0$$
in probability when $R \rightarrow \infty$. Using Slutsky's theorem and the fact that $K_R/K$ tends to $1$ in probability, one deduces that 
$$\frac{1}{K \pi} \underset{A \rightarrow \infty}{\lim}\int_{[-A,A]} \frac{d M(\rho)}{ s- \rho} $$
converges in distribution to a standard Cauchy variable, when $s$ is uniformly distributed on $(-R/2, R/2)$ and independent of $M$, and $R$ tends to infinity. 
Now, translation invariance of $M$ in distribution shows that for deterministic $s$, $h \in \mathbb{R}$,
$$ \frac{1}{K \pi} \underset{A \rightarrow \infty}{\lim}\int_{[-A,A]} \frac{d M(\rho)}{ s- \rho}
= \frac{1}{K \pi} \underset{A \rightarrow \infty}{\lim}\int_{[-A+h,A+h]} \frac{d M(\rho)}{ s+h- \rho}$$
in distribution. Now, replacing the interval $[-A+h, A+h]$ by $[-A,A]$ changes the last integral by at most 
$$ \frac{M ( [A, A+h] ) +M ( [-A, -A+h] }{(A - |s| - |h|)_+}.$$
The expectation of the numerator is bounded by a quantity which does not depend on $A$, whereas the denominator goes to infinity with $A$, so the variation of the integral tends to zero in probability when $A \rightarrow \infty$. Hence, by Slutsky's lemma, 
$$ \frac{1}{K \pi} \underset{A \rightarrow \infty}{\lim}\int_{[-A+h,A+h]} \frac{d M(\rho)}{ s+h- \rho} =
 \frac{1}{K \pi} \underset{A \rightarrow \infty}{\lim}\int_{[-A ,A ]} \frac{d M(\rho)}{ s+h- \rho}
$$
in distribution, and the distribution of 
$$  \frac{1}{K \pi} \underset{A \rightarrow \infty}{\lim}\int_{[-A,A]} \frac{d M(\rho)}{ s- \rho}$$
does not depend on $s \in \mathbb{R}$ if $s$ is deterministic. We get the same distribution if $s$ is independent of $M$ and uniform in $(-R/2, R/2)$ for any $R \geq 1$. This distribution is then necessarily standard Cauchy by taking the limit $R \rightarrow \infty$.

\subsection{Proof of Theorem \ref{Cauchy-Riemann}}

We consider the Riemann $\xi$  function given by
$$\xi(s) = \frac{s(s-1)}{2 \pi^{s/2}} \Gamma(s/2) \zeta(s),$$
which implies 
$$\frac{\xi'(s)}{\xi(s)} = \frac{1}{s} + \frac{1}{s-1} - \frac{1}{2} \log \pi + \frac{1}{2}\frac{ \Gamma' (s)}{ \Gamma(s)} + \frac{\zeta'(s)}{ \zeta(s)}. $$
Applying this formula to $s = 1/2 + it$, for $t \in [T/\log T, T]$, $T > 10$, we get 
$$\frac{2i }{\log T} \frac{\xi'(1/2 + it )}{\xi(1/2 + it)} = \frac{2i }{\log T} \frac{\zeta'(1/2 + it )}{\zeta(1/2 + it)} 
+ \frac{i}{ \log T}\frac{ \Gamma' (1/2 + it)}{ \Gamma(1/2 + it )} + \mathcal{O} (1/ \log T).$$
 Since 
 $$\frac{ \Gamma' (1/2 + it)}{ \Gamma(1/2 + it )}  = \log t + \mathcal{O}(1),$$
 we get 
 $$\frac{2i }{\log T} \frac{\xi'(1/2 + it )}{\xi(1/2 + it)} = \frac{2i }{\log T} \frac{\zeta'(1/2 + it )}{\zeta(1/2 + it)}  + i + \mathcal{O}(\log \log T/ \log T),$$
 which connects the logarithmic derivative of $\zeta$ to the logarithmic derivative of $\xi$.

 Since $\omega T \in [T/\log T, T]$ with probability going to $1$ when $T \rightarrow \infty$, it is sufficient to prove that 
 $$\left(\frac{2i}{\log T} \frac{\xi'(1/2 + i \omega T)}{\xi (1/2 + i \omega T)} \right)_{T \geq 2}
 $$
 converges to a standard Cauchy variable when $T$ tends to infinity. 
The function $\xi$ is real on the critical line, which implies that 
 $$\frac{2i }{\log T} \frac{\xi'(1/2 + i \omega T )}{\xi(1/2 + i \omega T)}$$
 is a random variable on $\mathbb{R} \cup \{\infty\}$, considered as the Alexandroff compactification of $\mathbb{R}$. By Prokhorov's theorem, it is enough to show that for all increasing sequences $(T_n)_{n \geq 1}$ such that 
 $$\frac{2i }{\log T} \frac{\xi'(1/2 + i \omega T_n)}{\xi(1/2 + i \omega T_n)}$$
 converges to a limiting distribution on $\mathbb{R} \cup \{\infty\}$, this distribution is standard Cauchy. 
  The sequence of random measures 
 $$M_{T_n} := \sum_{u \in \mathbb{R}, \, \xi(1/2 + i ( \omega T_n + 2 \pi u/\log T_n)) = 0} m(1/2 + i ( \omega T_n + 2 \pi u/\log T_n)) \delta_{u},$$
 where $m (\rho)$ denotes the multiplicity of a zero $\rho$ of $\xi$, is tight for the topology of vague convergence of locally finite measures. 
 Indeed, for fixed $A > 0$, the expectation of the measure of the interval $[-A,A]$ is 
 \begin{align*}
 & \int_0^1 d \tau \sum_{u \in [-A,A], \, \xi(1/2 + i ( \tau T_n + 2 \pi u/\log T_n)) = 0} m(1/2 + i ( \tau T_n + 2 \pi u/\log T_n)) 
 \\ & = \sum_{\gamma \in \mathbb{R}, \, \xi(1/2 + i \gamma) = 0} m(1/2 + i \gamma)  \int_0^{1} d \tau \mathds{1}_{|\gamma - \tau T_n| (\log T_n)/ 2 \pi \leq A}. 
  \end{align*} 
 The values of $\tau$ such that the last indicator function is equal to one form an interval of length at most $4 \pi A / (T_n \log T_n)$. Moreover, if this interval is non-empty, we have
 $\gamma \geq - 2 \pi A / \log T_n$ and $\gamma \leq T_n +  2 \pi A / \log T_n$. Hence, the expectation of the measure of $[-A,A]$ is at most 
 $$ \frac{4 \pi A}{T_n \log T_n}  \sum_{\gamma \in  [-2 \pi A/ \log T_n, T_n +  2 \pi A / \log T_n], \, \xi(1/2 + i \gamma) = 0} m(1/2 + i \gamma),$$
 which, by standard results on the distribution of the zeros of $\zeta$, is uniformly bounded in $n$ for fixed $A$. 
 This is enough for tightness of the sequence of random measures $(M_{T_n})_{n \geq 1}$
 since for all $\varepsilon \in (0,1)$, they are, with probability at least $1- \varepsilon$, in the set of measures $\mu$ on $\mathbb{R}$ such that 
 $$\{ \forall A \in \mathbb{N}, \mu([-A,A]) \leq C_{\varepsilon,A} \},$$
 for some $C_{\varepsilon,A}$ depending on $\varepsilon$ and $A$ but not on $n$. This set of measures is relatively compact for the topology of vague convergence.
 
The tightness property proven above shows that there necessarily exists a subsequence $(U_k = T_{n_k})_{k \geq 1}$, such that the following convergence in distribution holds, for the topology of vague convergence of locally finite measures on $\mathbb{R}$:  
$$M_{U_k} = \sum_{u \in \mathbb{R}, \xi(1/2 + i ( \omega U_k + 2 \pi u/\log U_k)) = 0} m(1/2 + i ( \omega U_k + 2 \pi u/\log U_k)) \delta_{u}  \underset{k \rightarrow \infty}{\longrightarrow} M$$
where $m (\rho)$ is the multiplicity of a zero $\rho$ of $\xi$, and $M$ is a random point measure. 

Moreover, the distribution of $M$ is necessarily invariant by translation. Indeed, $M$ is the limit of 
 $$\sum_{u \in \mathbb{R}, \, \xi(1/2 + i ( \omega U_k + 2 \pi u/\log U_k)) = 0} m(1/2 + i ( \omega U_k + 2 \pi u/\log U_k)) \delta_{u},$$
 whereas for a given $h \in \mathbb{R}$, the image of $M$ by translation of $h$ is the limit of 
 $$\sum_{u \in \mathbb{R}, \, \xi(1/2 + i ( \omega U_k + 2 \pi u/\log U_k)) = 0} m(1/2 + i ( \omega U_k + 2 \pi u/\log U_k)) \delta_{u+h},$$
 i.e. by a change of variable,
  $$\sum_{u \in \mathbb{R}, \, \xi(1/2 + i ( \omega' U_k + 2 \pi u/\log U_k)) = 0} m(1/2 + i ( \omega' U_k + 2 \pi u/\log U_k)) \delta_{u},$$
 where 
 $$\omega' = \omega - 2 \pi h / (U_k \log U_k). $$
  It is possible to couple the distributions of $\omega$ and $\omega'$ in such a way that the corresponding variables are equal to each other with probability $1 - \mathcal{O}( |h|/ (U_k \log U_k))$. This provides a similar coupling of 
  $$\sum_{u \in \mathbb{R}, \, \xi(1/2 + i ( \omega U_k + 2 \pi u/\log U_k)) = 0} m(1/2 + i ( \omega U_k + 2 \pi u/\log U_k)) \delta_{u}$$
  and 
  $$\sum_{u \in \mathbb{R}, \, \xi(1/2 + i ( \omega U_k + 2 \pi u/\log U_k)) = 0} m(1/2 + i ( \omega U_k + 2 \pi u/\log U_k)) \delta_{u+h},$$
  which implies the translation-invariance of the distribution of $M$.

 We now consider the random meromorphic function
 $$z \mapsto  \frac{2i \pi }{\log U_k} \frac{\xi'(1/2 + i ( \omega U_k + 2 \pi z/\log U_k) )}{\xi(1/2 + i( \omega U_k + 2 \pi z/\log U_k) )}.$$
 
 From \cite{S18}, we know that, under the Riemann hypothesis,  this function is in the Herglotz (i.e. Nevanlinna) class, i.e. the images of points in the upper-half plane are in the upper-half plane,  and that it converges in distribution, in the sense detailed in \cite{AW15}, to the random Herglotz function $F$
with poles, counted with multiplicity, following the point measure $M$,  and tending to $i \pi$ at infinity along the imaginary axis.

For $h \in \mathbb{R}$, the function $z \mapsto F(z-h)$ has limit $i \pi$ at $h + i \infty$, and then the same limit at $i \infty$, from general properties of Herglotz functions (in particular, the spectral measure is always integrable with respect to the Cauchy distribution, which allows a suitable application of dominated convergence). The spectral measure of 
$z \mapsto F(z-h)$ is the image of $M$ by translation by $h$, which has the same law as $M$. Hence, $F$ is invariant in distribution by real shift of the argument. 
Moreover, since $M$ is a  point measure (sum of Dirac measures), the limit of the imaginary part of $F$ at a given real point is almost sure equal to zero, which allows to apply Theorem 5.5 of \cite{AW15}, using the second remark following that theorem. We deduce that the distribution of $F$ at a given real point is a Cauchy distribution of analytic barycenter $i \pi$, i.e. with density $x \mapsto 1/(x^2+ \pi^2)$.

 \bibliographystyle{plain}
\bibliography{bibliography}

\end{document}